\documentclass[11pt,a4paper]{amsart}
\usepackage{amssymb,amsthm,enumitem,amsmath,mathrsfs,comment}
\usepackage[noadjust]{cite}
\usepackage{hyperref}
\usepackage[capitalise]{cleveref}
\usepackage{color}
\usepackage{kbordermatrix}

\theoremstyle{plain}
\newtheorem{theorem}{Theorem}[section]
\newtheorem{lemma}[theorem]{Lemma}
\newtheorem{corollary}[theorem]{Corollary}
\newtheorem{proposition}[theorem]{Proposition}

\newtheorem{claim}{}[theorem] 
\newenvironment{subproof}{\begin{proof}[Subproof.]}{\end{proof}}

\theoremstyle{definition}

\Crefname{theorem}{Theorem}{Theorems}
\crefrangeformat{claim}{#3#1#4--#5#2#6}
\crefmultiformat{claim}{#2#1#3}{ and~#2#1#3}{, #2#1#3}{ and~#2#1#3}
\crefrangemultiformat{claim}{#3#1#4--#5#2#6}{ and~#3#1#4--#5#2#6}{, #3#1#4--#5#2#6}{ and~#3#1#4--#5#2#6}

\DeclareMathOperator{\cl}{cl}

\DeclareMathOperator{\si}{si}
\DeclareMathOperator{\co}{co}
\newcommand{\del}{\backslash}
\newcommand{\ba}{\backslash}
\newcommand{\cocl}{\cl^*}

\newcommand{\dY}{$\Delta$\nobreakdash-$Y$}
\newcommand{\Yd}{$Y$\nobreakdash-$\Delta$}

\newcommand{\GF}{\textrm{GF}}
\newcommand{\seq}[1]{[#1]}
\newcommand{\utfutf}{\{U_{2,5},U_{3,5}\}}

\setenumerate{label=\rm(\roman*),midpenalty=2}

\usepackage[labelformat=simple]{subcaption}
\usepackage{tikz}
\usepackage{tkz-graph}
\tikzstyle{VertexStyle} = [shape = circle, draw, fill]
\tikzset{pre/.style={-}}    
\tikzstyle{every node}=[circle, inner sep=0pt, minimum width=4pt]
\usetikzlibrary{backgrounds}
\title{Excluded minors are almost fragile II: essential elements}

\author[N.\ Brettell]{Nick Brettell}
\address{School of Mathematics and Statistics, Victoria University of Wellington, New Zealand}
\email{nick.brettell@vuw.ac.nz}
\author[J.\ Oxley]{James Oxley}
\address{Department of Mathematics, Louisiana State University, Baton Rouge, Louisiana, USA}
\email{oxley@math.lsu.edu}
\author[C.\ Semple]{Charles Semple}
\address{School of Mathematics and Statistics, University of Canterbury, New Zealand}
\email{charles.semple@canterbury.ac.nz}
\author[G.\ Whittle]{Geoff Whittle}
\address{School of Mathematics and Statistics, Victoria University of Wellington, New Zealand}
\email{geoff.whittle@vuw.ac.nz}

\thanks{The first author was supported by a Rutherford Foundation postdoctoral fellowship.  The first, third and fourth authors were supported by the New Zealand Marsden Fund.}

\date{\today}

\keywords{matroids, excluded minors, representations, fragility}

\begin{document}

\begin{abstract}
  Let $M$ be an excluded minor for the class of $\mathbb{P}$-representable matroids for some partial field $\mathbb{P}$, let $N$ be a $3$-connected strong $\mathbb{P}$-stabilizer that is non-binary, and suppose $M$ has a pair of elements $\{a,b\}$ such that $M\ba a,b$ is $3$-connected with an $N$-minor.
  Suppose also that $|E(M)| \ge |E(N)|+11$ and $M \ba a,b$ is not $N$-fragile.
  In the prequel to this paper, we proved that $M \ba a,b$ is at most five elements away from an $N$-fragile minor.
  An element $e$ in a matroid $M'$ is \emph{$N$-essential} if neither $M'/e$ nor $M' \ba e$ has an $N$-minor.
  In this paper, we prove that, under mild assumptions, $M \ba a,b$ is one element away from a minor having at least $r(M)-2$ elements that are $N$-essential.
\end{abstract}

\maketitle

\section{Introduction}

This paper is a sequel to \cite{BCOSW20}.
It achieves a technical improvement on the results of that paper.
This technical improvement is a key step that enables us to prove that an excluded minor for the class of $2$-regular matroids, or the class of matroids with six inequivalent representations over $\GF(5)$, has at most 15 elements.
This makes it possible to compute the full list of excluded minors for $2$-regular matroids, and to bring the task of computing such a list within reach for the class of matroids with six inequivalent representations over $\GF(5)$ \cite{paper2}.
These results are important steps towards finding an excluded-minor characterisation for the class of matroids representable over all fields of size at least four, and for the class of $\GF(5)$-representable matroids.

It is unlikely that a reader would take an interest in this paper without being aware of the overall context that motivates this paper.
Our unlikely reader is referred to the introductions to \cite{BCOSW20,paper2,BWW20}.
We would be similarly surprised if the reader were not familiar with the technical terms in what follows.
Such an unexpected reader is referred to Section~\ref{pre} for definitions of these terms.

In \cite{BCOSW20} we showed that if
$M$ is an excluded minor for representability over a partial field~$\mathbb{P}$, and $M$ has a strong stabilizer~$N$ as a minor, then
\begin{enumerate}
  \item there is some matroid $M'$ that is $\Delta$-$Y$ equivalent to $M$, and $M'$ has a pair of elements $\{a,b\}$ such that $M \ba a,b$ is $N$-fragile; and
  \item for any delete pair $\{a,b\}$, the matroid $M \ba a,b$ is at most $5$ 
    elements away from being $N$-fragile.
\end{enumerate}
  
These results were intended as tools to bound the size of an excluded minor.
But they have some shortcomings.
For (i), we have to move to a $\Delta$-$Y$ equivalent matroid and cannot choose the delete pair.
This limits its applicability as we would like to be able to optimise our delete pair relative to other criteria.
For (ii), being seven elements away from a fragile minor negatively affects the bounds that we would like to obtain.
Ideally we would like to be able to say that for any delete pair $\{a,b\}$, the matroid $M \ba a,b$ is $N$-fragile.
Although we are still unable to do this in general, in this paper we circumvent the issue by considering elements that are ``$N$-essential'' in $M \ba a,b$ when this matroid is close to, but not, $N$-fragile.
  
Specifically, an element $e$ of a matroid $M'$ with an $N$-minor is \emph{$N$-essential} if neither $M'\backslash e$ nor $M'/e$ has an $N$-minor.
Evidently $M'$ can have at most $|E(N)|$ elements that are $N$-essential.
Thus finding $N$-essential elements in $M\backslash a,b$ gives us a means of controlling the size of $M$, even when $M\backslash a,b$ is not $N$-fragile.
\Cref{thegrandfantasy}, the main result of this paper, does precisely this.
It is a technical result, but it does the job, and we are able to use it in~\cite{paper2} to obtain the bounds discussed at the start of this introduction.
  
\section{Preliminaries} \label{pre}

All undefined matroid terminology follows Oxley~\cite{Oxley11}.
For sets $X$ and $Y$, we say $X$ \emph{meets} $Y$ if $X \cap Y \neq \emptyset$, and $X$ \emph{avoids} $Y$ if $X \cap Y = \emptyset$.
For a partition $\{X_1,X_2,\dotsc,X_m\}$ or an ordered partition $(X_1,X_2,\dotsc,X_m)$ we require that each cell $X_i$ is non-empty.

\subsection*{Segments, cosegments, and fans}

Let $M$ be a matroid.
A subset~$S$ of $E(M)$ with $|S| \ge 3$ is a \emph{segment} if every $3$-element subset of $S$ is a triangle.
A \emph{cosegment} is a segment of $M^*$.
A subset~$F$ of $E(M)$ with $|F| \ge 3$ is a \emph{fan} if there is an ordering $(f_1, f_2, \dotsc, f_\ell)$ of $F$ such that
\begin{itemize}
  \item[(a)] $\{f_1,f_2,f_3\}$ is either a triangle or a triad, and
  \item[(b)] for all $i \in \seq{\ell-3}$, if $\{f_i, f_{i+1}, f_{i+2}\}$ is a triangle, then $\{f_{i+1}, f_{i+2}, f_{i+3}\}$ is a triad, whereas if $\{f_i, f_{i+1}, f_{i+2}\}$ is a triad, then $\{f_{i+1}, f_{i+2}, f_{i+3}\}$ is a triangle.
\end{itemize}
When there is no ambiguity, we also say that the ordering $(f_1,f_2,\dotsc,f_\ell)$ is a fan.
If $F$ has a fan ordering $(f_1, f_2, \dotsc, f_\ell)$ where $\ell \geq 4$, then $f_1$ and $f_\ell$ are the \emph{ends} of $F$, and $f_2, f_3, \dotsc, f_{\ell-1}$ are the \emph{internal elements} of $F$.
We also say such a fan has \emph{size} $\ell$.
We say that a fan $F$ is \emph{maximal} if there is no fan that properly contains $F$.

\subsection*{Connectivity}

Let $M$ be a matroid and let $X \subseteq E(M)$.
The set $X$ or the partition $(X,E(M)-X)$ is \emph{$k$-separating} if $\lambda_M(X) < k$, where $$\lambda_M(X) = r(X) + r(E(M)-X) - r(M) = r(X)+r^*(X)-|X|.$$
A $k$-separating set $X$ or partition $(X,E(M)-X)$ is \emph{exact} if $\lambda_M(X) = k-1$.
If $X$ is $k$-separating and $|X|,|E(M)-X| \ge k$, then $X$ is a \emph{$k$-separation}.
The matroid $M$ is $k'$-connected if $M$ has no $k$-separations for $k < k'$.
If $M$ is $2$-connected, we simply say it is \emph{connected}.
Suppose $M$ is connected.
If for every $2$-separation $(X,Y)$ of $M$ either $X$ or $Y$ is a parallel pair (or parallel class), then $M$ is \emph{$3$-connected up to parallel pairs} (or \emph{parallel classes}, respectively).
Dually, if for every $2$-separation $(X,Y)$ of $M$ either $X$ or $Y$ is a series pair (or series class), then $M$ is \emph{$3$-connected up to series pairs} (or \emph{series classes}, respectively).

We say $Z \subseteq E(M)$ is in the \emph{guts} of a $k$-separation $(X,Y)$ if $Z \subseteq \cl(X-Z) \cap \cl(Y-Z)$, and we say $Z$ is in the \emph{coguts} of $(X,Y)$ if $Z$ is in the guts of $(X,Y)$ in $M^*$.
We also say $z$ is in the guts (or the coguts) of a $k$-separation $(X,Y)$ if $\{z\}$ is in the guts (or the coguts, respectively) of $(X,Y)$.
Note that if $z$ is in the guts of $(X,Y)$, then $z \notin \cocl(X)$ and $z \notin \cocl(Y)$.

We say that a partition $(X_1,X_2,\dotsc,X_m)$ of $E(M)$ is a \emph{path of $k$-separations} if $(X_1 \cup \dotsm \cup X_i, X_{i+1} \cup \dotsm \cup X_m)$ is exactly $k$-separating for each $i \in \seq{m-1}$.
Note that $|X_1|,|X_m| \ge 2$ (and $|X_i| \ge 1$ for all $i \in [m]$). 

We require the following lemma:
\begin{lemma}[see {\cite[Lemma~2.11]{BS14}}, for example]
  \label{gutsandcoguts}
  Let $(X,Y)$ be a $3$-separation of a $3$-connected matroid $M$.
  If $X \cap \cl(Y) \neq \emptyset$ and $X \cap \cocl(Y)\neq \emptyset$, then $|X \cap \cl(Y)|=1$ and $|X \cap \cocl(Y)|=1$.
\end{lemma}

A $3$-separation $(X,Y)$ of $M$ is a \textit{vertical $3$-separation} if $\min\{r(X),r(Y)\}\geq 3$. We also say that a partition $(X,\{z\},Y)$ is a \textit{vertical $3$-separation} of $M$ when both $(X\cup z,Y)$ and $(X,Y\cup z)$ are vertical $3$-separations with $z$ in the guts. We will write $(X,z,Y)$ for $(X,\{z\},Y)$.
If $(X,z,Y)$ is a vertical $3$-separation of $M$, then we say that $(X,z,Y)$ is a \emph{cyclic $3$-separation} of $M^*$.

Suppose $e \in E(M)$, and $(X,Y)$ is a partition of $M \ba e$ with $\lambda_{M \ba e}(X)=k$.
We say that $e$ \emph{blocks} $X$ if $\lambda_{M}(X) > k$.
If $e$ blocks $X$, then $e \notin \cl_M(Y)$.
We say that $e$ \emph{fully blocks} $(X,Y)$ if both $\lambda_M(X) > k$ and $\lambda_M(X \cup e) > k$.  It is easy to see that $e$ fully blocks $(X,Y)$ if and only if $e \notin \cl_M(X) \cup \cl_M(Y)$.

We say a set $T \subseteq E(M)$ is a \emph{triangle-triad} if $T$ is both a triangle and a triad.
Note that when $M$ has a triangle-triad $T$, we have $\lambda_M(T) = 1$, so if $|E(M)| \ge 5$, then $M$ is not $3$-connected.

\subsection*{Minors and fragility}

Let $M$ be a matroid, let $\mathcal{N}$ be a set of matroids, and let $x$ be an element of $M$.
For a matroid $N$, we say that \emph{$M$ has an $N$-minor} if $M$ has a minor isomorphic to $N$.
We say $M$ has an $\mathcal{N}$-minor if $M$ has an $N$-minor for some $N \in \mathcal{N}$.
If $M\ba x$ has an $\mathcal{N}$-minor, then $x$ is $\mathcal{N}$-\textit{deletable}.
If $M/x$ has an $\mathcal{N}$-minor, then $x$ is $\mathcal{N}$-\textit{contractible}.
If neither $M\ba x$ nor $M/x$ has an $\mathcal{N}$-minor, then $x$ is $\mathcal{N}$-\textit{essential}.
If $x$ is both $\mathcal{N}$-deletable and $\mathcal{N}$-contractible, then we say that $x$ is \textit{$\mathcal{N}$-flexible}.
A matroid $M$ is \textit{$\mathcal{N}$-fragile} if $M$ has an $\mathcal{N}$-minor, and no element of $M$ is $\mathcal{N}$-flexible
(note that sometimes this is referred to in the literature as ``strictly $\mathcal{N}$-fragile'').
For $X \subseteq E(M)$, we also say that $X$ is \emph{$\mathcal{N}$-deletable} (or \emph{$\mathcal{N}$-contractible}) when $M \del X$ (or $M / X$, respectively) has an $\mathcal{N}$-minor.
When $\mathcal{N} = \{N\}$, we use the prefix ``$N$-'' for these terms, rather than ``$\{N\}$-''.

The next lemma is well known, and the subsequent lemma is straightforward.

\begin{lemma}[see {\cite[Proposition~4.3]{MvZW10}}, for example]
  \label{minor3conn}
  Let $M$ be a matroid with a $2$-separation $(X,Y)$, and let $N$ be a $3$-connected minor of $M$.
  Then $|U \cap E(N)| \le 1$ for some $\{U,V\} = \{X,Y\}$.
  Moreover, 
  \begin{enumerate}
    \item if $u \in U - \cl(V)$, then $u$ is $N$-contractible; and
    \item if $u \in U - \cocl(V)$, then $u$ is $N$-deletable.
  \end{enumerate}
\end{lemma}

\begin{lemma}
  \label{niceVertSep}
  Let $(X, z, Y)$ be a vertical $3$-separation of a $3$-connected matroid $M$, and let $N$ be a $3$-connected minor of $M/z$. Then there exists a vertical $3$-separation $(X', z, Y')$ such that $|X' \cap E(N)| \le 1$ and $Y' \cup z$ is closed.
\end{lemma}

We also use the following well-known property of fragile matroids, which follows from \cref{minor3conn}.

\begin{lemma}[see {\cite[Proposition~4.4]{MvZW10}}, for example]
  \label{genfragileconn}
  Let $\mathcal{N}$ be a non-empty set of $3$-connected matroids with $|E(N)| \ge 4$ for each $N \in \mathcal{N}$.
  If $M$ is $\mathcal{N}$-fragile, then $M$ is $3$-connected up to series and parallel classes.
\end{lemma}

\subsection*{Representation theory}

A \textit{partial field} is a pair $(R, G)$, where $R$ is a commutative ring with unity, and $G$ is a subgroup of the group of units of $R$ such that $-1 \in G$.
If $\mathbb{P}=(R,G)$ is a partial field, then we say $p$ is in $\mathbb{P}$, and write $p\in \mathbb{P}$, whenever $p\in G\cup \{0\}$.

Let $\mathbb{P}$ be a partial field, and let $A$ be an $X\times Y$ matrix with entries from $\mathbb{P}$, where $X$ and $Y$ are disjoint sets.
Then $A$ is a $\mathbb{P}$-\textit{matrix} if every subdeterminant of $A$ is in $\mathbb{P}$.
If $X'\subseteq X$ and $Y'\subseteq Y$, then we write $A[X',Y']$ to denote the submatrix of $A$ induced by $X'$ and $Y'$.
When $Z\subseteq X\cup Y$, we denote by $A[Z]$ the submatrix induced by $X\cap Z$ and $Y\cap Z$, and we denote by $A-Z$ the submatrix induced by $X-Z$ and $Y-Z$.
 
\begin{theorem}[{\cite[Theorem 2.8]{PvZ10b}}]
\label{pmatroid}
Let $\mathbb{P}$ be a partial field, and let $A$ be an $X\times Y$ $\mathbb{P}$-matrix, where $X$ and $Y$ are disjoint sets. Let
\begin{equation*}
  \mathcal{B}=\{X\}\cup \{X\triangle Z : |X\cap Z|=|Y\cap Z| \textrm{ and } \det(A[Z])\neq 0\}. 
\end{equation*}
 Then $\mathcal{B}$ is the set of bases of a matroid on $X\cup Y$.
\end{theorem}

We say that the matroid in \cref{pmatroid} is $\mathbb{P}$-\textit{representable}, and that $A$ is a $\mathbb{P}$-\textit{representation} of $M$. We write $M=M[I|A]$ if $A$ is a $\mathbb{P}$-matrix, and $M$ is the matroid whose bases are described in \cref{pmatroid}. 

Let $\mathbb{P}$ be a partial field, let $A$ be an $X\times Y$ $\mathbb{P}$-matrix, for disjoint sets $X$ and $Y$, and let $x\in X$ and $y\in Y$ such that $A_{xy}\neq 0$. Then we define $A^{xy}$ to be the $(X\triangle \{x,y\})\times (Y\triangle \{x,y\})$ $\mathbb{P}$-matrix given by 
\begin{displaymath}
  (A^{xy})_{uv} =
\begin{cases}
    A_{xy}^{-1} \quad & \textrm{if } uv = yx\\
    A_{xy}^{-1} A_{xv} & \textrm{if } u = y, v\neq x\\
    -A_{xy}^{-1} A_{uy} & \textrm{if } v = x, u \neq y\\
    A_{uv} - A_{xy}^{-1} A_{uy} A_{xv} & \textrm{otherwise.}
\end{cases}
\end{displaymath}
We say that $A^{xy}$ is obtained from $A$ by \textit{pivoting} on $xy$.

Two $\mathbb{P}$-matrices are \textit{scaling equivalent} if one can be obtained from the other by repeatedly scaling rows and columns by non-zero elements of $\mathbb{P}$. Two $\mathbb{P}$-matrices are \textit{geometrically equivalent} if one can be obtained from the other by a sequence of the following operations: scaling rows and columns by non-zero entries of $\mathbb{P}$, permuting rows, permuting columns, and pivoting.
 
Let $\mathbb{P}$ be a partial field, and let $M$ and $N$ be matroids such that $N$ is a minor of $M$. Suppose that the ground set of $N$ is $X'\cup Y'$, where $X'$ is a basis of $N$. We say that $M$ is $\mathbb{P}$-\textit{stabilized by $N$} if, whenever $A_1$ and $A_2$ are $X\times Y$ $\mathbb{P}$-matrices with $X'\subseteq X$ and $Y'\subseteq Y$ such that
\begin{enumerate}
 \item[(i)] $M=M[I|A_1]=M[I|A_2]$,
 \item[(ii)] $A_1[X',Y']$ is scaling equivalent to $A_2[X',Y']$, and
 \item[(iii)] $N=M[I|A_1[X',Y']]=M[I|A_2[X',Y']],$ 
\end{enumerate}
then $A_1$ is scaling equivalent to $A_2$.
If $M$ is $\mathbb{P}$-stabilized by $N$, and every $\mathbb{P}$-representation of $N$ extends to a $\mathbb{P}$-representation of $M$, then we say $M$ is \emph{strongly $\mathbb{P}$-stabilized} by $N$.

Let $\mathcal{M}$ be a class of matroids. We say that $N$ is a \textit{$\mathbb{P}$-stabilizer for $\mathcal{M}$} if, for every $3$-connected $\mathbb{P}$-representable matroid $M\in \mathcal{M}$ with an $N$-minor, $M$ is $\mathbb{P}$-stabilized by $N$.
We say that $N$ is a \textit{strong $\mathbb{P}$-stabilizer for $\mathcal{M}$} if, for every $3$-connected $\mathbb{P}$-representable matroid $M\in \mathcal{M}$ with an $N$-minor, $M$ is strongly $\mathbb{P}$-stabilized by $N$.
Usually, we will be interested in the class of $\mathbb{P}$-representable matroids for some partial field~$\mathbb{P}$.
When $\mathcal{M}$ is the class of all $\mathbb{P}$-representable matroids, we
simply say ``$N$ is a strong $\mathbb{P}$-stabilizer''. 

\subsection*{Certifying non-representability}

Let $\mathbb{P}$ be a partial field.
Let $M$ be a matroid and let $E(M)=X \cup Y$ where $X$ and $Y$ are disjoint sets.
Let $A$ be an $X \times Y$ matrix with entries in $\mathbb{P}$ such that, for some distinct $a, b \in Y$, both $A-a$ and $A-b$ are $\mathbb{P}$-matrices, $M \del a=M[I|A-a]$, and $M \del b=M[I|A-b]$. 
Then we say $A$ is an $X \times Y$ \emph{companion $\mathbb{P}$-matrix} for $M$.

Let $B$ be a basis of $M$.
We write $B^*$ to denote $E(M)-B$.
Let $A$ be a $B\times B^*$ matrix with entries in $\mathbb{P}$. A subset~$Z$ of $E(M)$ \textit{incriminates} the pair $(M, A)$ if $A[Z]$ is square and one of the following holds: 
\begin{enumerate}
 \item[(i)] $\det(A[Z])\notin \mathbb{P}$,
 \item[(ii)] $\det(A[Z])=0$ but $B\triangle Z$ is a basis of $M$, or
 \item[(iii)] $\det(A[Z])\neq 0$ but $B\triangle Z$ is dependent in $M$.
\end{enumerate}

The next lemma follows immediately. 

\begin{lemma}
  Let $M$ be a matroid, let $A$ be an $X\times Y$ matrix with entries in a partial field $\mathbb{P}$, where $X$ and $Y$ are disjoint sets, and $X\cup Y=E(M)$. Exactly one of the following statements holds:
\begin{itemize}
 \item[(i)] $A$ is a $\mathbb{P}$-matrix and $M=M[I | A]$, or
 \item[(ii)] there is some $Z\subseteq X\cup Y$ that incriminates $(M, A)$.
\end{itemize}
\end{lemma}

Let $M$ be an excluded minor for the class of $\mathbb{P}$-representable matroids, for a partial field $\mathbb{P}$.
We will obtain a $B \times B^*$ companion $\mathbb{P}$-matrix $A$ for $M$ such that $\{x,y,a,b\}$ incriminates $(M,A)$ for some distinct $x,y \in B$ and $a,b \in B^*$.
In this setting, for $p \in B$ and $q \in B^*$ where $A_{pq} \neq 0$, we say that the pivot $A^{pq}$ is \emph{allowable} if $\{p,q\} \cap \{x,y,a,b\} \neq \emptyset$ and $\{x,y,a,b\} \triangle \{p,q\}$ incriminates $(M,A^{pq})$, or $\{p,q\} \cap \{x,y,a,b\} = \emptyset$ and $\{x,y,a,b\}$ incriminates $(M,A^{pq})$.
The next two lemmas describe situations where a pivot is allowable.

\begin{lemma}[{\cite[Lemma 5.10]{MvZW10}}]
  \label{allowablexyrow}
  Let $M$ be an excluded minor for the class of $\mathbb{P}$-representable matroids, for a partial field $\mathbb{P}$, and let $A$ be a $B\times B^{*}$ companion $\mathbb{P}$-matrix for $M$.
  Suppose that $\{x,y,a,b\}$ incriminates $(M,A)$, for pairs $\{x,y\}\subseteq B$ and $\{a,b\}\subseteq B^{*}$.
  If $p\in \{x,y\}$, $q\in B^{*}-\{a,b\}$, and $A_{pq}\neq 0$, then $A^{pq}$ is an allowable pivot.
\end{lemma}

\begin{lemma}[{\cite[Lemma 5.11]{MvZW10}}]
  \label{allowablenonxy}
  Let $M$ be an excluded minor for the class of $\mathbb{P}$-representable matroids, for a partial field $\mathbb{P}$, and let $A$ be a $B\times B^{*}$ companion $\mathbb{P}$-matrix for $M$.
  Suppose that $\{x,y,a,b\}$ incriminates $(M,A)$, for pairs $\{x,y\}\subseteq B$ and $\{a,b\}\subseteq B^{*}$.
  If $p\in B-\{x,y\}$, $q\in B^{*}-\{a,b\}$, $A_{pq}\neq 0$, and either $A_{pa}=A_{pb}=0$ or $A_{xq}=A_{yq}=0$, then $A^{pq}$ is an allowable pivot.
\end{lemma}

\subsection*{Delta-wye exchange}

Let $M$ be a matroid with a triangle $T=\{a,b,c\}$.
Consider a copy of $M(K_4)$ having $T$ as a triangle with $\{a',b',c'\}$ as the complementary triad labelled such that $\{a,b',c'\}$, $\{a',b,c'\}$ and $\{a',b',c\}$ are triangles.
Let $P_{T}(M,M(K_4))$ denote the generalised parallel connection of $M$ with this copy of $M(K_4)$ along the triangle $T$.
Let $M'$ be the matroid $P_{T}(M,M(K_4))\backslash T$ where the elements $a'$, $b'$ and $c'$ are relabelled as $a$, $b$, and $c$ respectively.
The matroid $M'$ is said to be obtained from $M$ by a \emph{\dY\ exchange} on the triangle~$T$, and is denoted $\Delta_T(M)$.
Dually, $M''$ is obtained from $M$ by a \emph{\Yd\ exchange} on the triad $T^*=\{a,b,c\}$ if $(M'')^*$ is obtained from $M^*$ by a \dY\ exchange on $T^*$.  The matroid $M''$ is denoted $\nabla_{T^*}(M)$.

We say that a matroid $M_1$ is \emph{\dY-equivalent} to a matroid $M_0$ if $M_1$ can be obtained from $M_0$ by a sequence of \dY\ and \Yd\ exchanges on coindependent triangles and independent triads, respectively.
We let $\Delta^*(M)$ denote the set of matroids that are \dY-equivalent to $M$ or $M^*$.

Oxley, Semple, and Vertigan proved that the set of excluded minors for $\mathbb{P}$-representability is closed under \dY\ exchange.
\begin{proposition}[{\cite[Theorem~1.1]{OSV00}}]
  \label{osvdelta}
  Let $\mathbb{P}$ be a partial field, and let $M$ be an excluded minor for the class of $\mathbb{P}$-representable matroids.
  If $M' \in \Delta^*(M)$, then $M'$ is an excluded minor for the class of $\mathbb{P}$-representable matroids.
\end{proposition}

The following result seems well known but, as we were unable to find a reference, we include a proof for completeness.

\begin{lemma}
  \label{dyconn}
  Let $M$ be a $3$-connected matroid with a coindependent triangle $T$, and let $M'$ be the \dY\ exchange of $M$ on $T$.
  Then $M'$ is $3$-connected up to series pairs.
  Moreover, $M'$ is not $3$-connected if and only if $T$ is contained in a $4$-element fan of $M$.
\end{lemma}
\begin{proof}
  Let $T = \{a,b,c\}$, and let $M'' = P_T(M, M(K_4))$ where the $M(K_4)$ has $T$ as a triangle and $\{a',b',c'\}$ as the complementary triad labelled such that $\{a,b',c'\}$, $\{a',b,c'\}$ and $\{a',b',c\}$ are triangles.
  Then $M' \cong M'' \ba T$.
  The matroid $M''$ is $3$-connected (see, for example, \cite[Section~11.4, Exercise~9]{Oxley11}), and there is no triad of $M''$ that meets $T$.
  Observe also that a cocircuit of $M''$ that contains $T$ also contains two elements of $\{a',b',c\}$, and at least one element of $E(M)-T$.  So $T$ is not contained in a $4$- or $5$-element cocircuit.
  By Bixby’s Lemma~\cite{Bixby82} (see also \cite[Lemma~8.7.3]{Oxley11}), and since $a$ is not in a triad of $M''$, the matroid $M'' \ba a$ is $3$-connected.
  By another application of Bixby's Lemma, $M'' \ba a,b$ is $3$-connected up to series pairs.

  If $M'' \ba a,b$ is $3$-connected, then $M'' \ba T$ is $3$-connected up to series pairs, by Bixby's Lemma.
  So assume that $M'' \ba a,b$ has a series pair $S$.
  Then $S \cup \{a,b\}$ is a cocircuit of $M''$ where $S=\{x,c'\}$ for some $x \in E(M) - T$.
  It follows that $S$ is the unique series pair in $M'' \ba a,b$.
  Now $M'' \ba a,b/c'$ is $3$-connected so, by Bixby's Lemma again, $M'' \ba T/c'$ is $3$-connected up to series pairs.
  Since $c'$ is not a loop in $M'' \ba T$, the matroid $M'' \ba T$ is $3$-connected up to series pairs.

  Suppose $M'' \ba T$ is not $3$-connected.  Then $M'' \ba T$ has a series pair $S'$ such that $S' \cup X$ is a $4$-element cocircuit of $M''$ for some pair $X \subseteq \{a,b,c\}$.  Without loss of generality, say $X=\{a,b\}$.  Then $S' = \{x,c'\}$ for some $x \in E(M)-T$.
  It follows that $\{x,a,b\}$ is a triad of $M$, so $T \cup x$ is a $4$-element fan of $M$.

  For the converse, it is easy to see that if $T$ is contained in a $4$-element fan, then $M'$ has a series pair.
\end{proof}

\begin{corollary}
  \label{no4fans}
  Let $\mathbb{P}$ be a partial field, and let $M$ be an excluded minor for the class of $\mathbb{P}$-representable matroids.
  Then $M$ is $3$-connected and has no $4$-element fans.
\end{corollary}
\begin{proof}
  It is well known, and easy to see, that when $M$ is an excluded minor for the class of $\mathbb{P}$-representable matroids, $M$ is $3$-connected.  
  Let $T$ be a triangle of $M$ that is contained in a $4$-element fan, and let $M'$ be the $\Delta$-$Y$ exchange of $M$ on $T$.
  Then $M'$ has a series pair, by \cref{dyconn}.
  But $M'$ is an excluded minor for the class of $\mathbb{P}$-representable matroids, by \cref{osvdelta}, a contradiction.
  So $M$ has no $4$-element fans.
\end{proof}

\subsection*{Delete pairs}

Let $M$ be a $3$-connected matroid, and let $N$ be a $3$-connected minor of $M$.
For distinct $a,b \in E(M)$, we say that $\{a,b\}$ is a \emph{delete pair for $N$} if $M \ba a,b$ is $3$-connected and has an $N$-minor.
Brettell, Whittle, and Williams \cite{BWW20,BWW21,BWW22} proved that if $|E(M)| - |E(N)| \ge 10$, then, for some $M'$ that is \dY-equivalent to $M$, either $M'$ has a delete pair for $N$, or $(M')^*$ has a delete pair for $N^*$, or $M'$ has a $3$-separation where one side is ``spike-like''.
Moreover, it was shown in \cite{BCOSW20} that such a $3$-separation, where one side is ``spike-like'', cannot appear in an excluded minor for the class of $\mathbb{P}$-representable matroids, for any partial field $\mathbb{P}$.

\begin{theorem}[{\cite[Corollary~7.3]{BCOSW20}}]
  \label{detachsetup}
Let $\mathbb{P}$ be a partial field, let $M$ be an excluded minor for the class of $\mathbb{P}$-representable matroids, and let $N$ be a non-binary $3$-connected strong stabilizer for the class of $\mathbb{P}$-representable matroids, where $M$ has an $N$-minor.
Suppose that $|E(M)| \ge |E(N)| + 10$.
Then there exist $M_0 \in \Delta(M)$ and $(M',N') \in \{(M_0,N), (M_0^*,N^*)\}$ such that $M'$ has a pair of elements $\{a,b\}$ for which $M' \ba a,b$ is $3$-connected and has an $N'$-minor.
%
\end{theorem}

\subsection*{Robust and strong elements, and bolstered bases}

Let $M$ be a $3$-connected matroid, let $B$ be a basis of $M$, and let $N$ be a $3$-connected minor of $M$.
Recall that we write $B^*$ to denote $E(M)-B$.
An element $e \in E(M)$ is \textit{$(N,B)$-robust} if either
\begin{enumerate}
 \item[(i)] $e\in B$ and $M/e$ has an $N$-minor, or
 \item[(ii)] $e\in B^*$ and $M\del e$ has an $N$-minor.
\end{enumerate}

\noindent
Note that an $N$-flexible element of $M$ is clearly $(N,B)$-robust for any basis~$B$ of $M$. 
An element $e \in E(M)$ is \textit{$(N,B)$-strong} if either
\begin{enumerate}
 \item[(i)] $e\in B$, and $\si(M/e)$ is $3$-connected and has an $N$-minor; or
 \item[(ii)] $e\in B^*$, and $\co(M\del e)$ is $3$-connected and has an $N$-minor.
\end{enumerate}

Now let $\{a,b\}$ be a pair of elements of $M$ such that $M \ba a,b$ is $3$-connected with an $N$-minor.
Let $B$ be a basis of the matroid $M \ba a,b$, and let $A$ be a $B\times B^{*}$ companion $\mathbb{P}$-matrix of $M$ such that $\{x,y,a,b\}$ incriminates $(M,A)$, for some $\{x,y\}\subseteq B$.
If either
  \begin{itemize}
    \item [(i)] $M \ba a,b$ has exactly one $(N,B)$-strong element $u$ outside of $\{x,y\}$, and $\{u,x,y\}$ is a triad of $M \ba a,b$; or
    \item [(ii)] $M \ba a,b$ has no $(N,B')$-strong elements outside of $\{x',y'\}$ for every choice of basis~$B'$ with a $B'\times (B')^{*}$ companion $\mathbb{P}$-matrix $A'$ of $M$ such that $\{x',y',a,b\}$ incriminates $(M,A')$, for some $\{x',y'\}\subseteq B'$;
  \end{itemize}
  then $B$ is a \emph{strengthened} basis.

In other words, a basis~$B$ is strengthened if $B$ is chosen such that either there is one $(N,B)$-strong element $u$ of $M \ba a,b$ outside of $\{x,y\}$, and $\{u,x,y\}$ is a triad; or there are no $(N,B)$-strong elements outside of $\{x,y\}$, and, moreover, there are no $(N,B')$-strong elements outside of $\{x',y'\}$ for any choice of basis~$B'$ with an incriminating set $\{x',y',a,b\}$ where $\{x',y'\} \subseteq B'$.
In particular, for a strengthened basis~$B$ with no $(N,B)$-strong elements, an allowable pivot cannot introduce an $(N,B)$-strong element.

Now suppose $B$ is strengthened.
We say that $B$ is \emph{bolstered} if
\begin{itemize}
  \item when $M \ba a,b$ has no $(N,B)$-strong elements outside of $\{x,y\}$, then
    for any $B_1\times B_1^{*}$ companion $\mathbb{P}$-matrix~$A_1$ where $\{x_1,y_1,a,b\}$ incriminates $(M,A_1)$, with $\{x_1,y_1\}\subseteq B_1$ and $\{a,b\}\subseteq B_1^{*}$, the number of $(N,B)$-robust elements of $M \ba a,b$ outside of $\{x,y\}$ is at least the number of $(N,B_1)$-robust elements of $M \ba a,b$ outside of $\{x_1,y_1\}$; or
  \item when $M \ba a,b$ has an $(N,B)$-strong element $u$ of $M \ba a,b$ outside of $\{x,y\}$, then for any $B_1\times B_1^{*}$ companion $\mathbb{P}$-matrix~$A_1$ such that 
    \begin{itemize}
      \item $\{x,y,a,b\}$ incriminates $(M,A_1)$, with $\{x,y\}\subseteq B_1$ and $\{a,b\}\subseteq B_1^{*}$, and
      \item $u$ is the only $(N,B_1)$-strong element of $M \ba a,b$, with $u \in B_1^*$,
    \end{itemize}
    the number of $(N,B)$-robust elements of $M \ba a,b$ is at least the number of $(N,B_1)$-robust elements of $M \ba a,b$. 
\end{itemize}
Loosely speaking, a strengthened basis~$B$ is bolstered if no allowable pivot increases the number of elements that are 
robust but not strong.

\subsection*{Excluded minors are almost fragile}

The results in the remainder of this section were proved in \cite{BCOSW20}, and are all under the following hypotheses.

\begin{itemize}[label=$(\spadesuit)$]
  \item Let $\mathbb{P}$ be a partial field.
  Let $M$ be an excluded minor for the class of $\mathbb{P}$-representable matroids, and let $N$ be a non-binary $3$-connected strong $\mathbb{P}$-stabilizer for the class of $\mathbb{P}$-representable matroids.
  Suppose $M$ has a pair of elements $\{a,b\}$ such that $M\ba a,b$ is $3$-connected with an $N$-minor. \label{hypotheses1}
\end{itemize}

\noindent
Henceforth we write $(\spadesuit)$ to remind the reader we are under these hypotheses.

\newcommand{\ha}{\hyperref[hypotheses1]{$(\spadesuit)$}}

  When \ha\ holds, note that, as $N$ is non-binary and $3$-connected, we have $|E(N)| \ge 4$ and $N$ is simple and cosimple.
  Thus, if $M'$ is a minor of $M$ that has an $N$-minor, and $M'$ has a parallel pair $\{e,f\}$, then $e$ and $f$ are $N$-deletable in $M$.
  We frequently use this fact, and its dual, without reference.

\begin{lemma}[{\cite[Lemma~3.1]{BCOSW20}}]
  \label{nostronginbasis}
  \ha\ Let $A$ be a $B \times B^*$ companion $\mathbb{P}$-matrix of $M$ such that $\{x, y, a, b\}$ incriminates $(M, A)$, where $\{x,y\} \subseteq B$ and $\{a,b\} \subseteq B^*$.
  If $v$ is an $(N, B)$-strong element of $M \ba a,b$ such that $v \notin \{x,y\}$, then $v \notin B$.
\end{lemma}

\begin{lemma}[{\cite[Proposition~3.7]{BCOSW20}}]
  \label{atmosttwostrong}
  \ha\ Let $A$ be a $B \times B^*$ companion $\mathbb{P}$-matrix of $M$ such that $\{x, y, a, b\}$ incriminates $(M, A)$, where $\{x,y\} \subseteq B$ and $\{a,b\} \subseteq B^*$.
  Then $M \ba a,b$ has at most two $(N,B)$-strong elements outside of $\{x,y\}$.
\end{lemma}

We say that a subset $G$ of $E(M\ba a,b)$ is a \textit{confining set} if $G\cap B_1=\{x_1,y_1\}$ for some basis~$B_1$ of $M\ba a,b$, and either
\begin{itemize}
  \item[(a)] $G$ is a $4$-cosegment, or
  \item[(b)] $G$ is the union of two triads $T$ and $T'$ with $|T \cap T'|=1$, where $G\cap B_1^{*}$ has at least one $(N,B_1)$-strong element,
\end{itemize}
where $x_1$ and $y_1$ are elements of $B_1$ such that $\{x_1,y_1,a,b\}$ incriminates $(M,A_1)$ for some $B_1\times B_1^{*}$ companion $\mathbb{P}$-matrix~$A_1$ of $M$.

\begin{lemma}[{\cite[Proposition~4.16]{BCOSW20}}]
  \label{confiningset}
  \ha\ If $M \ba a,b$ has a confining set, then $|E(M)| \le |E(N)|+9$.
\end{lemma}

\begin{lemma}
  \label{strongprops}
  \ha\ Suppose $|E(M)| \ge |E(N)| + 10$ and
  let $A$ be a $B \times B^*$ companion $\mathbb{P}$-matrix of $M$ such that $\{x, y, a, b\}$ incriminates $(M, A)$, where $\{x,y\} \subseteq B$ and $\{a,b\} \subseteq B^*$.
  If $v \in B^* - \{a,b\}$ is an $(N,B)$-strong element of $M\ba a,b$, 
  then
  \begin{enumerate}
    \item $v$ is in a triad $T^*$ of $M \ba a,b$ such that $\emptyset \subsetneqq T^* \cap B \subseteq \{x,y\}$, and
    \item for some $\{e,f\} = \{a,b\}$, the set $(T^*-v) \cup e$ is a triangle-triad of $M\ba f,v$.
  \end{enumerate}
\end{lemma}
\begin{proof}
  Let $v \in B^* - \{a,b\}$ be an $(N,B)$-strong element of $M\ba a,b$ such that
  $M \ba a,b,v$ has a series class of size at least three, so $v$ is in a cosegment $G$ of $M \ba a,b$ of size at least four.
  Then $G$ is a confining set by \cite[Lemma~3.2]{BCOSW20}, so $|E(M)| \le |E(N)|+9$ by \cref{confiningset}, a contradiction.
  So $M \ba a,b,v$ is $3$-connected up to series pairs, in which case the \lcnamecref{strongprops} holds by \cite[Lemma~3.4]{BCOSW20}.
\end{proof}

\begin{theorem}[{\cite[Theorem~6.7]{BCOSW20}}]
  \label{bcosw-thm}
  \ha\ Either
  \begin{itemize}
    \item[(i)] $|E(M)|\leq |E(N)|+9$, or 
    \item[(ii)] $M$ has a bolstered basis~$B$, and a $B\times B^{*}$ companion $\mathbb{P}$-matrix $A$ for which $\{x,y,a,b\}$ incriminates $(M,A)$, where $\{x,y\}\subseteq B$ and $\{a,b\}\subseteq B^{*}$, and either
      \begin{itemize} 
        \item[(a)] $M\del a,b$ is $N$-fragile, and $M\del a,b$ has at most one $(N,B)$-robust element outside of $\{x,y\}$, where if such an element $u$ exists, then $u\in B^{*}-\{a,b\}$ is an $(N,B)$-strong element of $M\del a,b$, and $\{u,x,y\}$ is a coclosed triad of $M\del a,b$, or
        \item[(b)] $M\del a,b$ is not $N$-fragile, but there is an element $u \in B^*-\{a,b\}$ that is $(N,B)$-strong in $M \ba a,b$; either
          \begin{itemize}
            \item[(I)] the $N$-flexible, and $(N,B)$-robust, elements of $M\del a,b$ are contained in $\{u,x,y\}$, or
            \item[(II)] the $N$-flexible, and $(N,B)$-robust, elements of $M\del a,b$ are contained in $\{u,x,y,z\}$, where $z \in B$, and $(z,u,x,y)$ is a maximal fan of $M \del a,b$, or
            \item[(III)]the $N$-flexible, and $(N,B)$-robust, elements of $M\del a,b$ are contained in $\{u,x,y,z,w\}$, where $z \in B$, $w \in B^*$, and $(w,z,x,u,y)$ is a maximal fan of $M \del a,b$;
          \end{itemize}
          the unique triad in $M \ba a,b$ containing $u$ is $\{u,x,y\}$; and $M$ has a cocircuit $\{x,y,u,a,b\}$ and a triangle $\{d,x,y\}$ for some $d\in \{a,b\}$.
      \end{itemize}
  \end{itemize}
\end{theorem}

\noindent
\begin{figure}[htbp]
  \begin{subfigure}{0.32\textwidth}
    \centering
    \begin{tikzpicture}[rotate=90,scale=0.8,line width=1pt]
      \tikzset{VertexStyle/.append style = {minimum height=5,minimum width=5}}
      \clip (-2.5,0.5) rectangle (2.5,-5);
      \draw (0,-1) .. controls (-2.75,1.5) and (-3.25,-2.5) .. (0,-4);
      \draw (0,-1) .. controls (3,-1.5) and (3,-4.5) .. (0,-4);
      \draw (0.6,-1.6) -- (1,-3.6);
      \draw (0,-1) -- (0,-4);
      \draw (1.4,-1.8) .. controls (0.8,-0.2) and (-0,0) .. (-1.0,-1.0);

      \Vertex[x=0.6,y=-1.6,L=$x$,Lpos=90,LabelOut=True]{x}
      \Vertex[x=1,y=-3.6,L=$y$,Lpos=90,LabelOut=True]{y}

      \tikzset{VertexStyle/.append style = {fill=white}}
      \Vertex[x=0.5,y=-0.5,L=$a$,Lpos=180,LabelOut=True]{a}
      \Vertex[x=0.8,y=-2.6,L=$b$,Lpos=90,LabelOut=True]{b}
      \Vertex[x=1.6,y=-3,L=$u$,Lpos=90,LabelOut=True]{u}

    \end{tikzpicture}
    \caption{A Type~I gadget~$(\{x,y\},u)$.}
  \end{subfigure}
  \begin{subfigure}{0.32\textwidth}
    \centering
    \begin{tikzpicture}[rotate=90,scale=0.8,line width=1pt]
      \tikzset{VertexStyle/.append style = {minimum height=5,minimum width=5}}
      \clip (-2.5,0.5) rectangle (2.5,-5);
      \draw (0,-1) .. controls (-2.75,1.5) and (-3.25,-2.5) .. (0,-4);
      \draw (0,-1) .. controls (3,-1.5) and (3,-4.5) .. (0,-4);
      \draw (0,-2.45) -- (1.9,-2.55) -- (0,-3.75);
      \draw (0.95,-3.15) -- (0,-1.25);
      \draw (1.4,-1.8) .. controls (0.8,-0.2) and (-0,0) .. (-1.0,-1.0);

      \draw (0,-1) -- (0,-4);

      \Vertex[Lpos=0,LabelOut=True,L=$x$,x=1.9,y=-2.55]{x}
      \Vertex[Lpos=180,LabelOut=True,L=$y$,x=0.62,y=-2.49]{y}
      \Vertex[Lpos=90,LabelOut=True,L=$z$,x=0,y=-3.75]{z}

      \tikzset{VertexStyle/.append style = {fill=white}}
      \Vertex[x=0.5,y=-0.5,L=$a$,Lpos=180,LabelOut=True]{a}
      \Vertex[Lpos=180,LabelOut=True,L=$b$,x=1.26,y=-2.52]{b}
      \Vertex[Lpos=0,LabelOut=True,L=$u$,x=0.95,y=-3.15]{u}

    \end{tikzpicture}
    \caption{A Type~II gadget~$(x,y,u,z)$.}
  \end{subfigure}
  \begin{subfigure}{0.32\textwidth}
    \centering
    \begin{tikzpicture}[rotate=90,scale=0.8,line width=1pt]
      \tikzset{VertexStyle/.append style = {minimum height=5,minimum width=5}}
      \clip (-2.5,0.5) rectangle (2.5,-5);
      \draw (0,-1) .. controls (-2.75,1.5) and (-3.25,-2.5) .. (0,-4);
      \draw (0,-1.25) -- (1.8,-1.25) -- (1.8,-3.75) -- (0,-3.75);
      \draw (1.8,-2.5) -- (0.9,-1.25);
      \draw (1.4,-1.5) .. controls (0.8,0.1) and (-0,0.3) .. (-1.0,-0.7);
      \draw[dashed] (1.8,-1.25) -- (0,-3.75);
      \draw[dashed] (1.8,-3.75) -- (0,-1.25);

      \draw (0,-1) -- (0,-4);

      \Vertex[Lpos=90,LabelOut=True,L=$x$,x=1.8,y=-2.5]{x}
      \Vertex[Lpos=180,LabelOut=True,L=$y$,x=0.9,y=-1.25]{y}
      \Vertex[Lpos=90,LabelOut=True,L=$z$,x=1.8,y=-3.75]{z}

      \tikzset{VertexStyle/.append style = {fill=white}}
      \Vertex[x=0.5,y=-0.2,L=$a$,Lpos=180,LabelOut=True]{a}
      \Vertex[Lpos=0,LabelOut=True,L=$b$,x=1.5,y=-2.083]{b}
      \Vertex[Lpos=90,LabelOut=True,L=$u$,x=1.8,y=-1.25]{u}
      \Vertex[Lpos=0,LabelOut=True,L=$w$,x=0.9,y=-3.75]{w}

    \end{tikzpicture}
    \caption{A Type~III gadget~$(x,y,u,z,w)$.}
  \end{subfigure}
  \caption{Illustrations of the three types of gadget when $M \ba a,b$ is not $N$-fragile, conforming with \cref{bcosw-thm}(ii)(b), 
    under the assumption that $a$ fully blocks in the gadget.
    Black elements are in the basis~$B$, whereas white elements are not in $B$.}
  \label{figobstrs}
\end{figure}
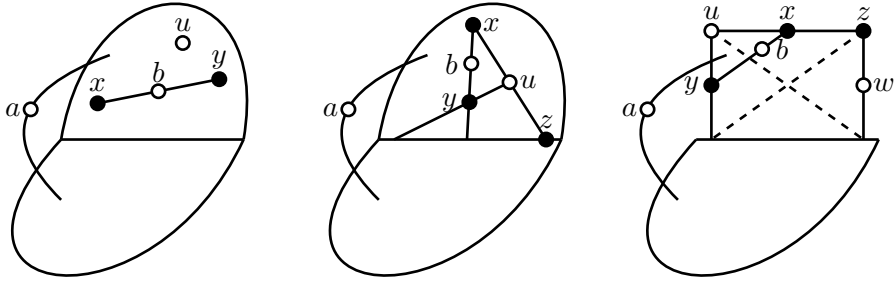

We require some terminology to handle the different outcomes when $M \ba a,b$ is not $N$-fragile.
Suppose \cref{bcosw-thm}(ii)(b) holds, so $M \ba a,b$ is not $N$-fragile.
Then we say that we have a \emph{gadget for $\{a,b\}$} that is either Type~I, Type~II, or Type~III:
\begin{itemize}
  \item if \cref{bcosw-thm}(ii)(b)(III) holds and $w$ is $(N,B)$-robust, then we say $(x,y,u,z,w)$ is a \emph{Type~III} gadget;
  \item if not, but \cref{bcosw-thm}(ii)(b)(II) or \cref{bcosw-thm}(ii)(b)(III) holds and $z$ is $(N,B)$-robust, then $(x,y,u,z)$ is a \emph{Type~II} gadget;
  \item otherwise, we say $(\{x,y\},u)$ is a \emph{Type~I} gadget.
\end{itemize}
Note, in particular, that if \cref{bcosw-thm}(ii)(b)(II) or \cref{bcosw-thm}(ii)(b)(III) holds but the $(N,B)$-robust elements (and hence also $N$-flexible elements) are contained in $\{x,y,u\}$, then we call this a Type~I gadget.
We also note that for a Type~III gadget $(x,y,u,z,w)$, the fan $(y,u,x,z,w)$ in maximal; whereas for a Type~II gadget $(x,y,u,z)$, either the fan $(y,x,u,z)$ is maximal, or it is contained in a maximal $5$-element fan $(y,u,x,z,w)$ where $w$ is not $(N,B)$-robust nor $N$-flexible.

Observe that for a Type~III gadget $(x,y,u,z,w)$, as $w$ is $N$-deletable in $M \ba a,b$, it follows that $x$ and $z$ are $N$-contractible in $M \ba a,b$; 
continuing in this way, each element in $\{x,u,z,w\}$ is $N$-flexible (whereas $y$ might only be $N$-contractible).
Similarly, for a Type~II gadget $(x,y,u,z)$, each element in $\{x,u,z\}$ is $N$-flexible (whereas $y$ might only be $N$-contractible).
For a Type~I gadget $(\{x,y\},u)$, the element $u$ is $N$-flexible in $M \ba a,b$ (whereas $x$ and $y$ might only be $N$-contractible);
to see this, observe that $u$ is $N$-deletable (as it is $(N,B)$-strong) and if $u$ is not $N$-contractible in $M \ba a,b$, then neither $x$ nor $y$ is $N$-deletable in $M \ba a,b$, so no element in $\{x,y,u\}$ is $N$-flexible, contradicting that $M \ba a,b$ is not $N$-fragile.
We summarise this in the next lemma.

\begin{lemma}
  \label{flexigadgets}
  \ha\ Suppose $|E(M)| \ge |E(N)| + 10$, and \cref{bcosw-thm}(ii)(b) holds, so $M \ba a,b$ is not $N$-fragile.  Suppose that either
  \begin{itemize}
    \item $(\{x,y\},u)$ is a Type~I gadget for $\{a,b\}$,
    \item $(x,y,u,z)$ is a Type~II gadget for $\{a,b\}$, or
    \item $(x,y,u,z,w)$ is a Type~III gadget for $\{a,b\}$.
  \end{itemize}
  Then $u$ is $N$-flexible in $M \ba a,b$, and $x$ and $y$ are $N$-contractible in $M \ba a,b$.
  Moreover, 
  \begin{itemize}
    \item if the gadget is Type~II, then $x$ and $z$ are $N$-flexible in $M \ba a,b$; and
    \item if the gadget is Type~III, then $x$, $z$, and $w$ are $N$-flexible in $M \ba a,b$.
  \end{itemize}
\end{lemma}

Let $G$ be a gadget for $\{a,b\}$.
When $G$ is a Type~I gadget $(\{x,y\},u)$, we say that $e$ is \emph{in} the gadget when $e \in \{x,y,u\}$; similarly, when $G$ is a Type~II gadget $(x,y,u,z)$ or a Type~III gadget $(x,y,u,z,w)$, we say that $e$ is \emph{in} the gadget when $e \in \{x,y,u,z\}$ or $e \in \{x,y,u,z,w\}$ respectively.

Suppose $G$ is a Type~I gadget $(\{x,y\},u)$, so $\{x,y,u\}$ is $3$-separating in $M \ba a,b$. 
Then either $b \in \cl(\{x,y\})$ and $a \in \cocl(\{x,y,u,b\})$, in which case we say $a$ \emph{blocks} in the gadget~$G$; or $a \in \cl(\{x,y\})$ and $b \in \cocl(\{x,y,u,b\})$, in which case we say $b$ \emph{blocks} in the gadget~$G$.
Suppose $a$ blocks in the gadget~$G$.
We say that $a$ \emph{fully blocks} in the gadget~$G$, or $a$ \emph{fully blocks} $(\{x,y\},u)$, if $a \notin \cl(\{x,y,u\})$; otherwise, $a$ \emph{does not fully block} $(\{x,y\},u)$ when $a \in \cl(\{x,y,u\})$.

Similarly, when $G$ is a Type~II gadget $(x,y,u,z)$ or a Type~III gadget $(x,y,u,z,w)$ with $b \in \cl(\{x,y\})$, and $a \in \cocl(\{x,y,u,z,b\})$ or $a \in \cocl(\{x,y,u,z,w,b\})$ respectively, we say that $a$ \emph{blocks} in the gadget~$G$;
and we say that $a$ \emph{fully blocks} in the gadget~$G$ when $a$ blocks and, moreover, $a \notin \cl(\{x,y,u,z\}$ or $a \notin \cl(\{x,y,u,z,w\}$ respectively.

\medskip

We also use the following:
\begin{lemma}[{\cite[Lemma~8.1]{BCOSW20}}]
  \label{oldwin}
  \ha\ Suppose that \cref{bcosw-thm}(ii)(b) holds.
  If $\{a,b\} \subseteq \cl_M(\{x,y\})$, then $|E(M)| \le |E(N)|+8$.
\end{lemma}

\begin{lemma}[{\cite[Lemma~8.2]{BCOSW20}}]
  \label{oldwin2}
  \ha\ Suppose that \cref{bcosw-thm}(ii)(b) holds and there is some $p \in (B - \{x, y\}) \cap \cl(\{u, x, y\})$ such that $\{a, b\} \subseteq \cl_M(\{p, x, y\})$. Then $|E(M)| \le |E(N)| +8$.
\end{lemma}

\section{Essential elements in almost-fragile minors}

In this section, we prove our main result, \cref{thegrandfantasy}.
We work under the following hypotheses throughout the remainder of this paper.
\begin{itemize}[label=$(\clubsuit)$]
  \item Let $\mathbb{P}$ be a partial field.
    Let $M$ be an excluded minor for the class of $\mathbb{P}$-representable matroids, and let $N$ be a non-binary $3$-connected strong $\mathbb{P}$-stabilizer for the class of $\mathbb{P}$-representable matroids, where $M$ has an $N$-minor.\label{hypotheses2}
\end{itemize}
  Loosely speaking, \cref{thegrandfantasy} says that if $M$ has no triads and is sufficiently large relative to $N$, 
  and there is a pair of elements $\{a,b\}$ in $M$ such that $M \ba a,b$ is $3$-connected with an $N$-minor,
  and $M \ba a,b$ is not $N$-fragile, then we can easily find elements that are $N$-essential in a minor of $M \ba a,b$ that is obtained by deleting or contracting a single element.

  First, we prove a lemma which shows that the requirement that $M$ has no triads is only a mild assumption, due to \cref{osvdelta}.
  Recall that $\Delta^*(M)$ denotes the class of matroids that are \dY-equivalent to $M$ or $M^*$.

\newcommand{\hb}{\hyperref[hypotheses2]{$(\clubsuit)$}}

\begin{lemma}
  \label{notriads}
  \hb\ Suppose that $|E(M)| \ge |E(N)| + 10$.
  Then there exists a matroid~$M_1 \in \Delta^*(M)$ such that $M_1$ has a pair of elements $\{a,b\}$ for which $M_1 \ba a,b$ is $3$-connected and has a $\Delta^*(N)$-minor, and $M_1$ has no triads.
\end{lemma}
\begin{proof}
  By \cref{detachsetup}, there is some $M_0 \in \Delta^*(M)$ such that $M_0$ has a pair of elements $\{a,b\}$ for which $M_0 \ba a,b$ is $3$-connected and has an $\{N,N^*\}$-minor.
  Let $T^*$ be a triad of $M_0$, and let $M_0'$ be obtained from $M_0$ by a \Yd\ exchange on $T^*$.
  Since $M_0 \ba a,b$ is $3$-connected, 
  $\cocl_{M_0}(T^*) \cap \{a,b\} = \emptyset$ and, moreover, $a \notin \cocl_{M_0}(T^* \cup b)$ and $b \notin \cocl_{M_0}(T^* \cup a)$.
  Hence $T^*$ is a triad of $M_0 \ba a,b$, and $M_0' \ba a,b$ is isomorphic to the matroid obtained from $M_0 \ba a,b$ by performing a \Yd\ exchange on $T^*$.
  In particular, note that $M_0' \ba a,b$ has a $\Delta^*(N)$-minor.

  Suppose $M_0' \ba a,b$ is not $3$-connected.
  Then $T^*$ is in a $4$-element fan~$F$ of $M_0 \ba a,b$, by \cref{dyconn}.
  But $a,b \in \cl_{M_0}(E(M_0 \ba a,b)-T^*)$.
  Thus $F$ is a $4$-element fan of $M_0$, contradicting \cref{no4fans}.
  So $M_0' \ba a,b$ is $3$-connected.

  It now follows that, starting from $M_0$ and repeatedly performing \Yd\ exchanges until no triads remain, we obtain a matroid $M_1 \in \Delta^*(M)$ with no triads, and having the property that $M_1 \ba a,b$ is $3$-connected and has a $\Delta^*(N)$-minor, as required.
\end{proof}

Our main result is under the assumption that $M \ba a,b$ is not $N$-fragile and the gadget for $\{a,b\}$ is Type~I.  We work towards \cref{wmatype1}, which shows that if the gadget for $\{a,b\}$ is not Type~I, and $a$ blocks in this gadget, then we can switch to the delete pair $\{a,x\}$ for $N$, in which case the gadget for $\{a,x\}$ is Type~I.

Recall that when \hb\ holds, $N$ is non-binary and $3$-connected, so $|E(N)| \ge 4$ and $N$ is simple and cosimple.
In particular, if $|E(M)| \ge |E(N)| + 10$, then $|E(M)| \ge 14$.

\begin{lemma}
  \label{switchbxy}
  \hb\ Suppose $M$ has a pair of elements $\{a,b\}$ such that $M\ba a,b$ is $3$-connected with an $N$-minor, $|E(M)| \ge |E(N)| + 10$, and $M \ba a,b$ is not $N$-fragile, so \cref{bcosw-thm}(ii)(b) holds.
  Assume the element~$a$ blocks in the gadget for $\{a,b\}$.
  Then $M \ba a,e$ is $3$-connected with an $N$-minor for each $e \in \{x,y\}$.
\end{lemma}
\begin{proof}
  Observe that $M \ba a$ is $3$-connected.
  Suppose $\{x,y\}$ is in a triad $\{x,y,t\}$ of $M \ba a$.
  Since $M \ba a,b$ is $3$-connected, $t \neq b$.
  Moreover, $t \neq u$ since $\{x,y,u,b\}$ is a cocircuit of $M \ba a$.
  So $\{x,y,u,t\}$ is a $4$-cosegment in $M \ba a,b$, contradicting that $\{x,y,u\}$ is the unique triad in $M \ba a,b$ containing $u$.
  Thus $M \ba a$ has no triad containing $\{x,y\}$.

  Now, if $M \ba a,x$ is not $3$-connected, then $M \ba a,x$ has a $2$-separation $(U,V)$ where $y \in V$ and $V$ is not a series class.
  It follows that $(\cocl(U),V-\cocl(U))$ is also a $2$-separation of $M \ba a,x$, so we may assume that $U$ is coclosed.
  Since $\{b,y,u\}$ is a triad in $M \ba a,x$, either $b$ or $u$ is in $V$.
  But if $b \in V$, then $x \in \cl_{M \ba a}(V)$, so $(U,V \cup x)$ is a $2$-separation of $M \ba a$, a contradiction.
  So $b \in U$ and $u \in V$.
  Now, $b \in \cl_{M \ba a}(V \cup x)$ and $b \in \cocl_{M \ba a}(V \cup x)$, where $V \cup x$ is $3$-separating in $M \ba a$, so $(U-b, V \cup \{x,b\})$ is a contradictory $2$-separation of $M \ba a$ unless $|U-b| = 1$.
  In the exceptional case, $U \cup x$ is a triad of $M \ba a$ that contains $b$, contradicting that $M \ba a,b$ is $3$-connected.
  We deduce that $M \ba a,x$ and, by symmetry, $M \ba a,y$ are $3$-connected.

  Finally, since $u$ is $(N,B)$-strong in $M \ba a,b$, the matroid $M \ba a,u$ has an $N$-minor, and $\{b,x,y\}$ is a triangle-triad in this matroid, so this set is $2$-separating.
  Thus $x$ and $y$ are $N$-contractible in $M \ba a,u$ and hence also in $M \ba a$, by \cref{minor3conn}.
  Hence $y$ and $x$ are $N$-deletable in $M \ba a$, as required.
\end{proof}

\begin{lemma}
  \label{wmatype1}
  \hb\ Suppose $M$ has a pair of elements $\{a,b\}$ such that $M\ba a,b$ is $3$-connected with an $N$-minor, $|E(M)| \ge |E(N)| + 10$, and $M \ba a,b$ is not $N$-fragile, so \cref{bcosw-thm}(ii)(b) holds.
  Assume the element~$a$ blocks in the gadget for $\{a,b\}$.
  Then, either
  \begin{enumerate}
    \item the gadget for $\{a,b\}$ is Type~I, or
    \item $M \ba a,x$ is $3$-connected with an $N$-minor, but not $N$-fragile, and the gadget for $\{a,x\}$ is Type~I.
  \end{enumerate}
\end{lemma}
\begin{proof}
  Suppose the lemma does not hold.
  Then, as $M \ba a,b$ is not $N$-fragile, the gadget for $\{a,b\}$ is Type~II or Type~III, otherwise (i) holds.
  So there is an element~$z$ in the gadget such that $\{x,u,z\}$ is a triangle, and $x$ is $N$-flexible in $M\ba a,b$ by \cref{flexigadgets}.
  Note also that $\{b,y,u,z\}$ is a circuit in $M$.

  By \cref{switchbxy}, $M \ba a,x$ is $3$-connected with an $N$-minor.
  \begin{claim}
    \label{switchnotfragile}
    $M \ba a,x$ is not $N$-fragile.
  \end{claim}
  \begin{subproof}
    First, we claim that $u$ and $b$ are $N$-flexible in $M \ba a,x$.
    Since $x$ is $N$-deletable in $M \ba a,b$, 
    and $u$ is in series in $M \ba a,b,x$, the matroid $M \ba a,b,x/u$ has an $N$-minor.
    In particular, $b$ is $N$-deletable and $u$ is $N$-contractible in $M \ba a,x$.
    Since $M \ba a,u$ has an $N$-minor, and $\{x,b,y\}$ is a triangle-triad in this matroid, $M \ba a,u,x/b$ has an $N$-minor.
    In particular, $b$ is $N$-contractible and $u$ is $N$-deletable in $M \ba a,x$.
    So $b$ and $u$ are $N$-flexible in $M \ba a,x$ as claimed.
    In particular, $M \ba a,x$ is not $N$-fragile.
  \end{subproof}

  Now, as $M \ba a,x$ is not $N$-fragile, the gadget for $\{a,x\}$ is Type~II or Type~III, otherwise (ii) holds.
  Let the gadget for $\{a,x\}$ be $(x',y',u',z')$ or $(x',y',u',z',w')$, if Type~II or Type~III respectively.
  In either case, $\{x',u',z'\}$ is a triangle in $M \ba a,x$.
  As $u$ and $b$ are $N$-flexible in $M \ba a,x$, we have $\{u,b\} \subseteq \{x',y',u',z',w'\}$.

  Suppose that the gadget is Type~II and $z' \in \{u,b\}$.
  Since $\{b,u,y\}$ is a triad of $M \ba a,x$, and $b$ and $u$ are $N$-flexible in this matroid, we have $\{b,u\} \subseteq \{x',u',z'\}$, and $y \notin \{y',x',u',z'\}$.
  Now $\{b,u\}$ intersects the triad $\{x',y',u'\}$ of $M \ba a,x$ in one element, so by orthogonality with the circuit $\{b,y,u,z\}$, we deduce that 
  $z \in \{x',y',u'\}$.
  But if $z \in \{x',u'\}$, then the triangle $\{x',u',z'\}$ is $\{b,u,z\}$, so $\{b,x,y\} \subseteq \cl_M(\{u,z\})$, in which case $\{x,y,u\}$ is a triangle-triad of $M \ba a,b$, a contradiction.
  So $y'=z$.
  Now $\{a,x,x',z,u'\}$ is a cocircuit of $M$, where $\{x',u'\}$ meets $\{b,u\}$.
  If $z' = b$, then $u \in \{x',u'\}$, and by orthogonality with the triangle $\{x,b,y\}$, we have $y \in \{x',u'\}$, a contradiction.
  So $z'=u$ and $b \in \{x',u'\}$.
  We claim that $\{x',y',u',z',x\}=\{x',z,u',u,x\}$ is $3$-separating in $M \ba a$.
  This is clear if $x \in \cl(\{x',y'\})$, so assume that $a \in \cl(\{x',y'\})$.
  Now $\{x',u'\} \subseteq \cl_M(\{b,u\})$, and hence $\{a,x\} \subseteq \cl_M(\{z,b,u\})$.
  So $x$ does not fully block $\{a,x',z,u',u\}$ in $M \ba x$, and hence  $\{x',z,u',u,x\}$ is $3$-separating in $M \ba a$, as claimed.
  But $y \in \cl(\{x',z,u',x,u\}) \cap \cocl_{M \ba a}(\{x',z,u',x,u\})$, implying $\lambda_{M \ba a}(\{x',z,u',x,u,y\}) = 1$ and $|E(M\ba a)| \le 7$, a contradiction.

  Now $z' \notin \{u,b\}$ or the gadget is Type~III.
  If the gadget is Type~II, then $\{u,b\}$ is contained in the triad $T^*=\{x',y',u'\}$ of $M\ba a,x$.
  If the gadget is Type~III, then $\{u,b\}$ is contained in a triad $T^*$ of $M \ba a,x$ such that $T^* \subseteq \{x',y',u',z',w'\}$, except in the case that $\{y',w'\} = \{u,b\}$.
  We return to this exceptional case later; for now we assume that $\{u,b\}$ is contained in a triad~$T^*$ that is contained in the gadget for $\{a,x\}$.
  Let $T^* = \{u,b,t\}$; then $\{u,x,t\}$ is a triad of $M\ba a,b$.
  Thus $t=y$.
  So $y$ is also in the gadget for $\{a,x\}$.
  Now, since $\{x,b,y\}$ is a triangle that intersects $\{a,x\}$ in the element~$x$, and $\{b,y\}$ is contained in the gadget for $\{a,x\}$, we have $\{b,y\} = \{x',y'\}$ and $x \in \cl(\{x',y'\})$.
  Then $u=u'$.
  Thus either $\{b, u, z'\}$ or $\{y, u, z'\}$ is a triangle.
  Now $r_M(\{x,u,y,b,z,z'\}) = 3$.
  Note also that $z' \neq z$, since neither $\{b, u, z\}$ nor $\{y, u, z\}$ is a triangle.
  Moreover, in the case that the gadget for $\{a,b\}$ is Type~III, then, letting $(x,y,u,z,w)$ be the gadget, we have $z' \neq w$, since neither $\{b, u, w\}$ nor $\{y, u, w\}$ is a triangle.
  Since $u$ is $N$-contractible in $M \ba a,b$, by \cref{flexigadgets}, the element~$z'$ is $N$-deletable in $M\ba a,b$.
  Thus, $z' \in B$.
  But $B$ contains $\{x,y,z,z'\}$, four elements of a rank-$3$ set, a contradiction.

  Now we may assume that $\{y',w'\} = \{u,b\}$.
  In this case, $\{a,b,x,y,u\}$ is a cocircuit of $M$ that intersects $\{u',x',z'\}$ in at most one element, $y$; so, by orthogonality, $y \notin \{u',x',z'\}$.
  Suppose $b=y'$.
  Then $M \ba a,b,x$ has $\{u',x'\}$ as a series pair, so $u'$ and $x'$ are $N$-contractible in $M \ba a,b,x$.
  Since $\{u',x',z'\}$ is a triangle, it follows that $u'$, $x'$, and $z'$ are $N$-deletable in $M \ba a,b,x$.
  Since $x'$ is $N$-deletable, $z'$ is also $N$-contractible in $M \ba a,b,x$.
  Now $\{u',x',z'\}$ are $N$-flexible elements of $M \ba a,b$, disjoint from $\{x,y,u\}$, contradicting that there is a gadget for $\{a,b\}$.
  A similar argument applies when $b = w'$.
\end{proof}

The following lemma is used several times in the proof of \cref{thegrandfantasy}.

\begin{lemma}
  \label{easywin}
  \hb\ Suppose $M$ has a pair of elements $\{a,b\}$ such that $M\ba a,b$ is $3$-connected with an $N$-minor, $|E(M)| \ge |E(N)| + 10$, and $M \ba a,b$ is not $N$-fragile, so \cref{bcosw-thm}(ii)(b) holds.
  Assume the element~$a$ blocks in the gadget for $\{a,b\}$.
  Then $\cl_M(\{b,x,y\}) = \{b,x,y\}$.
\end{lemma}
\begin{proof}
  Suppose that $e \in \cl_{M \ba a}(\{b,x,y\}) - \{b,x,y\}$.
  As $M \ba a,b$ is $3$-connected, $e \neq u$, and, if the gadget for $\{a,b\}$ is Type~II or III, $e$ is not in the gadget, in either case.
  Since $M \ba a,b/x$ has an $N$-minor, and $\{e,y\}$ is a parallel pair in this matroid, $e$ is $N$-deletable in $M \ba a,b$.
  Since $e$ is not in the gadget for $\{a,b\}$, the element~$e$ is not $N$-flexible or $(N,B)$-robust, so $e \in B$.  But now $\{x,y,e\}$ is a rank-$2$ set containing three elements of a basis~$B$ of $M$, a contradiction.

  It remains only to show that $a \notin \cl_M(\{b,x,y\})$.
  But if $a \in \cl_M(\{b,x,y\})$, then, by \cref{oldwin2}, $|E(M)| \le |E(N)| + 8$, a contradiction.
\end{proof}

We now come to our main result.

\begin{theorem}
  \label{thegrandfantasy}
  \hb\ Suppose $M$ has a pair of elements $\{a,b\}$ such that $M\ba a,b$ is $3$-connected with an $N$-minor, $M$ has no triads, $|E(M)| \ge |E(N)| + 11$, and $M \ba a,b$ is not $N$-fragile, so \cref{bcosw-thm}(ii)(b) holds.
  If the gadget for $\{a,b\}$ is Type~I, then,
  for every $b' \in B-\{x,y\}$, 
  the element $b'$ is $N$-essential in at least one of $M \ba a,b,u$ and $M \ba a,b/u$.
\end{theorem}
\begin{proof}
  Since \cref{bcosw-thm}(ii)(b) holds, $M$ has a bolstered basis~$B$ and a $B\times B^{*}$ companion $\mathbb{P}$-matrix $A$ for which $\{x,y,a,b\}$ incriminates $(M,A)$, where $\{x,y\}\subseteq B$ and $\{a,b\}\subseteq B^{*}$; and there is an element $u \in B^*-\{a,b\}$ that is $(N,B)$-strong in $M \ba a,b$.
  Assume that the gadget for $\{a,b\}$ is $(\{x,y\},u)$.
  Up to swapping $a$ and $b$, we may assume that $a$ blocks in the gadget, so $\{b,x,y\}$ is a triangle.

Pick $b' \in B-\{x,y\}$.  Then $b'$ is 
not $N$-contractible in $M \ba a,b$ (it is either $N$-essential or $N$-deletable), since the $(N,B)$-robust elements are contained in $\{x,y,u\}$.

\begin{claim}
  \label{setup}
  Either $b'$ is $N$-essential in one of $M \ba a,b,u$ and $M \ba a,b/u$, or $u$ is $N$-flexible in 
  $M \ba a,b,b'$.
\end{claim}
\begin{subproof}
  Suppose $b'$ is not $N$-essential in both $M \ba a,b,u$ and $M \ba a,b/u$.
  As $b'$ is not $N$-contractible in $M \ba a,b$, we have that $b'$ is $N$-deletable in $M \ba a,b,u$ and $M \ba a,b / u$. Hence $M \ba a,b,b',u$ 
and $M \ba a,b,b' / u$ have $N$-minors.
\end{subproof}

By \cref{setup}, we may assume that $M \ba a,b,b',u$ and $M \ba a,b,b'/u$ have $N$-minors, for otherwise $b'$ is $N$-essential in $M \ba a,b,u$ or $M \ba a,b / u$, as required.
In particular, $b'$ is $N$-deletable in $M \ba a,b$.

Recall that $M$ is $3$-connected, by \cref{no4fans}.
We start with four claims regarding connectivity.

\begin{claim}
  \label{conn1}
  $M \ba a,b,b'$ is $3$-connected up to series classes.
  Moreover, there is at most one series pair in $M \ba a,b,b'$ that does not meet $\{x,y\}$.
\end{claim}
\begin{subproof}
  Suppose $M \ba a,b,b'$ is not $3$-connected up to series classes.
  Then, by the dual of \cref{niceVertSep}, $M \ba a,b$ has a cyclic $3$-separation $(U,b',V)$, where $V \cup b'$ is coclosed, and $|U \cap E(N)| \le 1$.
  Now $U \cap \cocl_{M \ba a,b}(V) = \emptyset$ and, by \cref{gutsandcoguts}, $|(U \cup b') \cap \cl(V)| \le 1$, so $|U \cap \cl(V)| \le 1$.
  Thus, by \cref{minor3conn}, each element in $U$ is $N$-deletable in $M \ba a,b,b'$, and at most one element in $U$ is not $N$-flexible in $M \ba a,b,b'$.
  Since $r^*(U) \ge 3$ and the set of $N$-flexible elements in $M \ba a,b$ is contained in $\{u,x,y\}$, we have $|U \cap \{u,x,y\}| \ge 2$ and $|U-\{u,x,y\}|=1$.
  Let $U' = U\cup \{u,x,y\}$ and $V' = V-U'$, and observe that $(U',b',V')$ is a cyclic $3$-separation, where $U' = \{x,y,u,z'\}$, and $z'$ is the unique element in $U'$ that is not $N$-flexible in $M \ba a,b,b'$.
  But then $z'$ is not $N$-contractible, $z' \in \cl(V')$, and hence, by orthogonality, $b' \in \cocl_{M \ba a,b}(\{x,y,u\})$, a contradiction.
  So $M \ba a,b,b'$ is $3$-connected up to series classes.

  Now suppose that $M \ba a,b,b'$ has distinct series pairs $\{s_1,t_1\}$ and $\{s_2,t_2\}$ that both avoid $\{x,y\}$.
  Then, for $i \in \{1,2\}$, the elements $s_i$ and $t_i$ are $N$-contractible in $M \ba a,b,b'$ and hence in $M\ba a,b$.  Since these elements avoid $\{x,y\}$, we have $s_i,t_i \in B^*$.
  By cocircuit elimination on the triads $\{s_1,t_1,b'\}$ and $\{s_2,t_2,b'\}$ of $M \ba a,b$, there is a cocircuit contained in $\{s_1,t_1,s_2,t_2\}$.  But this set is contained in $B^*$, so is coindependent, a contradiction.
\end{subproof}

\begin{claim}
  \label{conn2}
  $M \ba a,b'$ is $3$-connected up to series classes.
  Moreover, every series pair of $M \ba a,b'$ avoids $\{x,y\}$.
\end{claim}
\begin{subproof}
  Suppose $M \ba a,b'$ has a series pair $S$ meeting $\{x,y\}$.
  Without loss of generality, let $S = \{x,t\}$ for some $t$.
  Since $M$ has no triads, $\{x,t,a,b'\}$ is a cocircuit of $M$.
  Note that $t \neq y$, since otherwise $\{u,x,y,b'\}$ is a cosegment of $M \ba a,b$; and $t \neq b$, since $M \ba a,b$ is $3$-connected.
  But now the cocircuit $\{x,t,a,b'\}$ intersects the triangle $\{x,y,b\}$ of $M$ in a single element, contradicting orthogonality.

  Let $(U,V)$ be a $2$-separation of $M \ba a,b'$ where neither $U$ nor $V$ is a series class.
  Then $|U|,|V| \ge 3$.
  Without loss of generality, let $b \in V$.
  Then $(U,V-b)$ is a $2$-separation of $M \ba a,b,b'$, so, by \cref{conn1}, either $U$ or $V-b$ is a series class of $M \ba a,b,b'$.
  Since the $2$-separating set $U$ in $M \ba a,b,b'$ remains $2$-separating in $M \ba a,b'$, it is not blocked by $b$, so $b \in \cl_{M\ba a,b'}(V-b)$.
  If $U$ is a series class in $M \ba a,b,b'$, it is blocked by $b$ in $M \ba a,b'$, so $b \notin \cl_{M \ba a,b'}(V-b)$, a contradiction.

  So $V-b$ is a series class in $M \ba a,b,b'$ that is blocked by $b$ in $M \ba a,b'$.
  Note that $V$ is a cosegment in $M \ba a,b'$ that contains $b$.
  By orthogonality with the triangle $\{x,y,b\}$, we may assume that $x \in V$.
  If $y \in V$, then $\{x,y,b\}$ is a triad in $M \ba a,b'$, contradicting that $\{u,b,x,y\}$ is a cocircuit.
  It follows, by orthogonality, that
  $V = \{b,x,s\}$ for some $s$, and hence $r(V \cup y) = 2$.
  Now $\{x,s\}$ is a series pair in $M \ba a,b,b'$, so $M \ba a,b,b'/s$ has an $N$-minor.
  Since $\{x,y\}$ is a parallel pair in this matroid, $M \ba a,b,x$ has an $N$-minor.
  But $\{s,b'\}$ is a series pair in this matroid, so $b'$ is $N$-contractible in $M \ba a,b$, a contradiction.
\end{subproof}

\begin{claim}
  \label{conn3}
  Either
  \begin{enumerate}
    \item $M \ba b,b'$ is $3$-connected up to series pairs, and each series pair contains $x$ or $y$ but not both; or
    \item $M \ba b,b'$ has a triangle-triad $\{x,s,a\}$, for $s \in E(M \ba b,b') - \{a,u,x,y\}$, and $M \ba a,b'$ is $3$-connected.
  \end{enumerate}
\end{claim}
\begin{subproof}
  Let $S$ be a series pair of $M \ba b,b'$.
  Since $M$ has no triads, $S \cup \{b,b'\}$ is a cocircuit of $M$.
  By orthogonality, $S$ meets $\{x,y\}$.
  Note also that $\{x,y\}$ is not a series pair of $M \ba b,b'$, otherwise $\{x,y,u,b'\}$ is a contradictory cosegment of $M \ba a,b$.
  Hence every non-trivial series class of $M \ba b,b'$ has size two.

  Let $(U,V)$ be a $2$-separation of $M \ba b,b'$ where neither $U$ nor $V$ is a series class.
  Then $|U|,|V| \ge 3$.
  Without loss of generality, let $a \in V$.
  Then $(U,V-a)$ is a $2$-separation of $M \ba a,b,b'$, so, by \cref{conn1}, either $U$ or $V-a$ is a series class of $M \ba a,b,b'$.
  Since the $2$-separating set $U$ in $M \ba a,b,b'$ remains $2$-separating in $M \ba b,b'$, it is not blocked by $a$, so $a \in \cl_{M\ba b,b'}(V-a)$.
  If $U$ is a series class in $M \ba a,b,b'$, it is blocked by $a$ in $M \ba b,b'$, so $a \notin \cl_{M \ba b,b'}(V-a)$, a contradiction.

  So $V-a$ is a series class in $M \ba a,b,b'$ that is blocked by $a$ in $M \ba b,b'$.
  As $a \in \cl_{M\ba b,b'}(V-a)$, the element $a$ is in a circuit of $M \ba b,b'$ contained in $V$.
  Since $\{x,y,u,a\}$ is a cocircuit of $M \ba b,b'$, orthogonality implies that $V-a$ meets $\{x,y,u\}$.
  Observe that $(V-a) \cup b'$ is a cosegment in $M \ba a,b$.
  As $\{x,y,u\}$ is the unique triad of $M \ba a,b$ containing $u$, it follows that $u \notin V-a$ and $\{x,y\} \nsubseteq V-a$.
  Without loss of generality, $x \in V-a$ and $\{u,y\} \subseteq U$.

  By \cref{conn1}, the series class $V-a$ of $M \ba a,b,b'$ has size at most three, otherwise there is more than one series pair that does not meet $\{x,y\}$.
  Suppose $|V|=4$.
  Let $V=\{a,x,s,s'\}$.  
  Note that $\{x,s,s',b'\}$ is a cosegment in $M \ba a,b$; and $s,s' \notin B$, since these elements are $N$-contractible in $M \ba a,b,b'$ and hence also in $M \ba a,b$.
  Now $\{x,s,s'\}$ and $\{x,y,u\}$ are triads of $M \ba a,b$ that intersect in a single element, and $\{x,s,s',y,u\} \cap B = \{x,y\}$.  This is a confining set relative to $B$ so $|E(M)| \le |E(N)| + 9$ by \cref{confiningset}, a contradiction.

  So we may assume $V = \{a,x,s\}$ for some $s$, where $V$ is a triangle-triad of $M \ba b,b'$.
  Observe that $\{x,s,b'\}$ is a triad of $M \ba a,b$.
  Suppose $M \ba a,b'$ is not $3$-connected.
  Then \cref{conn2} implies that $M \ba a,b'$ has a series pair~$S'$, which avoids $\{x,y\}$.
  Since $M$ has no triads, $S' \cup a$ is a triad in $M \ba b'$.
  By orthogonality with the triangle $V$, we have $s \in S'$.
  Then, by orthogonality with the triangle $\{b,x,y\}$, we have $b \notin S'$.
  Let $S'=\{s,t\}$.
  Then $s$ and $t$ are $N$-contractible in $M \ba a,b,b'$ and hence in $M\ba a,b$.
  So $s,t \in B^*$.
  Now $\{s,t,b'\}$ is a triad in $M \ba a$, and also in $M \ba a,b$, since these matroids are $3$-connected.
  So $\{x,s,t,b'\}$ is a cosegment of $M \ba a,b$.
  In particular, $\{s,t,x\}$ and $\{x,y,u\}$ are triads whose union intersects $B$ in $\{x,y\}$.
  This is a confining set, so $|E(M)| \le |E(N)| + 9$ by \cref{confiningset}, a contradiction.
  So $M \ba a,b'$ is $3$-connected. 
\end{subproof}

\begin{claim}
  \label{conn5}
  $M \ba b'$ is $3$-connected.
\end{claim}
\begin{subproof}
  Suppose not.
  Let $(U,V)$ be a $2$-separation of $M \ba b'$.
  Since $M$ has no triads, neither $U$ nor $V$ is a series class.
  In particular, $|U|,|V| \ge 3$.
  Without loss of generality, let $a \in V$.
  Now $(U,V-a)$ is a $2$-separation of $M \ba a,b'$.
  By \cref{conn2}, either $U$ or $V-a$ is a series class that is blocked by $a$ in $M \ba b'$.
  But $U$ is not blocked by $a$, since it remains $2$-separating in $M \ba b'$.
  So $V-a$ is a series class, and hence, by \cref{conn2}, $\{x,y\} \subseteq U$.
  Note also that $b \in U$, since $b$ is not in a series pair in $M \ba a,b'$, otherwise $M \ba a,b$ is not $3$-connected.
  We now have that $(U-b,V)$ is a $2$-separation in $M \ba b,b'$.
  By \cref{conn3}, since $M \ba a,b'$ is not $3$-connected, $M \ba b,b'$ is $3$-connected up to series classes.
  So $U-b$ is a series class, which contains both $x$ and $y$.
  But then $\{x,y,b'\}$ is a triad in $M \ba a,b$, a contradiction.
\end{subproof}

We require a definition for one particularly annoying case.
Suppose that $M \ba a,b$ and $M \ba b,b'$ are $3$-connected but not $N$-fragile, where the gadget for $\{a,b\}$ is Type~I, and the gadget for $\{b,b'\}$ is either Type~I or Type~II.
Let $x,y,u,u'$ be distinct elements in $M \ba a,b,b'$.
Then, we say that $(x,y,u,u',a,b,b')$ is a \emph{mega-gadget} when
\begin{itemize}
  \item $(\{x,y\},u)$ is the gadget for $\{a,b\}$, where $b \in \cl(\{x,y\})$ and $a$ fully blocks in the gadget; and 
  \item the gadget for $\{b,b'\}$ is $(\{u,y\},u')$, if Type~I; or either $(u,y,u',a)$ or $(y,u,u',a)$, if Type~II; in either case, $b' \in \cl(\{u,y\})$ and $b$ fully blocks in the gadget. 
\end{itemize}

The next claim deals with the case where $M \ba b,b'$ is $3$-connected, except when a mega-gadget arises; we defer this subcase to \cref{megagadgetcase}.

\begin{claim}
  \label{caseanalysis}
  If $M \ba b,b'$ is $3$-connected, then either $u$ is not $N$-flexible in $M \ba a,b,b'$, or there exists an element $u'$ such that $(x,y,u,u',a,b,b')$ is a mega-gadget.
\end{claim}
\begin{subproof}
Suppose $M \ba b,b'$ is $3$-connected, but $u$ is $N$-flexible in $M \ba b,b'$.
Then $M \ba b,b'$ is not $N$-fragile, 
so $u$ is in the gadget for $\{b,b'\}$.
Let the gadget for $\{b,b'\}$ be $(\{x',y'\},u')$, $(x',y',u',z')$, or $(x',y',u',z',w')$, if Type~I, Type~II, or Type~III respectively.
The case where the gadget is Type~I is illustrated in \cref{fig-cases}.
Note that either:
\begin{enumerate}
  \item[\textit{(Case 1)}] $b'$ blocks in the gadget, so $b \in \cl_M(\{x',y'\})$ and $b' \in \cocl_M(\{b,x',y',u'\})$; or
  \item[\textit{(Case 2)}] $b$ blocks in the gadget, so $b' \in \cl_M(\{x',y'\})$ and $b \in \cocl_M(\{b',x',y',u'\})$.
\end{enumerate}

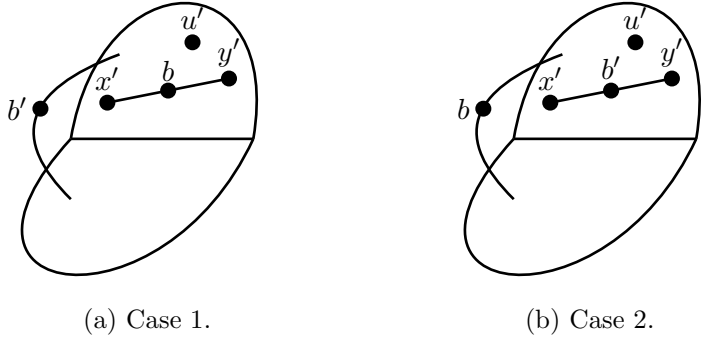
\begin{figure}[htbp]
  \begin{subfigure}{0.45\textwidth}
    \centering
    \begin{tikzpicture}[rotate=90,scale=0.8,line width=1pt]
      \tikzset{VertexStyle/.append style = {minimum height=5,minimum width=5}}
      \clip (-2.5,0.5) rectangle (2.5,-5);
      \draw (0,-1) .. controls (-2.75,1.5) and (-3.25,-2.5) .. (0,-4);
      \draw (0,-1) .. controls (3,-1.5) and (3,-4.5) .. (0,-4);
      \draw (0.6,-1.6) -- (1,-3.6);
      \draw (0,-1) -- (0,-4);
      \draw (1.4,-1.8) .. controls (0.8,-0.2) and (-0,0) .. (-1.0,-1.0);

      \Vertex[x=0.6,y=-1.6,L=$x'$,Lpos=90,LabelOut=True]{x}
      \Vertex[x=1,y=-3.6,L=$y'$,Lpos=90,LabelOut=True]{y}

      \Vertex[x=0.5,y=-0.5,L=$b'$,Lpos=180,LabelOut=True]{a}
      \Vertex[x=0.8,y=-2.6,L=$b$,Lpos=90,LabelOut=True]{b}
      \Vertex[x=1.6,y=-3,L=$u'$,Lpos=90,LabelOut=True]{u}

    \end{tikzpicture}
    \caption{Case 1.}
  \end{subfigure}
  \begin{subfigure}{0.45\textwidth}
    \centering
    \begin{tikzpicture}[rotate=90,scale=0.8,line width=1pt]
      \tikzset{VertexStyle/.append style = {minimum height=5,minimum width=5}}
      \clip (-2.5,0.5) rectangle (2.5,-5);
      \draw (0,-1) .. controls (-2.75,1.5) and (-3.25,-2.5) .. (0,-4);
      \draw (0,-1) .. controls (3,-1.5) and (3,-4.5) .. (0,-4);
      \draw (0.6,-1.6) -- (1,-3.6);
      \draw (0,-1) -- (0,-4);
      \draw (1.4,-1.8) .. controls (0.8,-0.2) and (-0,0) .. (-1.0,-1.0);

      \Vertex[x=0.6,y=-1.6,L=$x'$,Lpos=90,LabelOut=True]{x}
      \Vertex[x=1,y=-3.6,L=$y'$,Lpos=90,LabelOut=True]{y}

      \Vertex[x=0.5,y=-0.5,L=$b$,Lpos=180,LabelOut=True]{a}
      \Vertex[x=0.8,y=-2.6,L=$b'$,Lpos=90,LabelOut=True]{b}
      \Vertex[x=1.6,y=-3,L=$u'$,Lpos=90,LabelOut=True]{u}

    \end{tikzpicture}
    \caption{Case 2.}
  \end{subfigure}
  \caption{The gadget for $\{b,b'\}$.  Note that although here the gadget is illustrated as Type~I, and $b'$ (in case 1) or $b$ (in case 2) is fully blocking, the gadget might be Type~II or Type~III, and $b'$ or $b$ (respectively) might not fully block.}
  \label{fig-cases}
\end{figure}

First we consider case 1: when $b'$ blocks in the gadget and $b \in \cl(\{x',y'\})$.
Observe that when $u \in \{u',x',y'\}$, we have $a \notin \{x',y',u'\}$, otherwise $M \ba a,b$ has a triad containing $u$ distinct from $\{u,x,y\}$, since it contains $b'$.
We consider subcases based on which element is $u$ in the gadget for $\{b,b'\}$.

\textit{(Case 1.1)} Suppose $u=u'$.
In $M \ba b',u$, the set $\{b,x',y'\}$ is a triangle-triad.
Moreover, $\{b,x,y\}$ is a triangle in $M$, and hence in $M \ba b',u$.
By orthogonality, $\{x,y\}$ meets $\{x',y'\}$.
So $r(\{b,x,y,x',y'\}) = 2$.
By \cref{easywin}, $\{x,y\} = \{x',y'\}$.
Now $b' \in \cocl_M(\{x,y,u,b\})$, so $b' \in \cocl_{M \ba a,b}(\{x,y,u\})$, contradicting that $\{x,y,u\}$ is the unique triad containing $u$ in $M \ba a,b$.

\textit{(Case 1.2)} Suppose $u \in \{x',y'\}$.
In $M \ba b'$, the set $\{b,x',y',u'\}$ is a cocircuit, and $\{b,x,y\}$ is a circuit, so by orthogonality we may assume that $x \in \{x',y',u'\}$.
If $\{u,x\}=\{x',y'\}$, then as $\{x',y'\}$ spans $b$, the set $\{u,x\}$ spans $\{b,y\}$, so $r_{M \ba a,b}(\{u,x,y\})=2$, a contradiction.
So we may assume that $x = u'$.
Now $\cl(\{b,x,y,u\}) = \cl(\{b,x',y',u'\})$.
In $M \ba b'$, the set $\{b,x',y',x\}$ is $3$-separating, and has $y$ in its closure, so $\{b,x',y',x,y\}$ is $3$-separating in $M \ba b'$ and in $M \ba a,b'$.
By \cref{conn2,conn5}, it follows that $r^*_{M \ba b'}(\{b,x',y',x,y\}) = r^*_{M \ba a,b'}(\{b,x',y',x,y\}) = 4$, so $a \notin \cocl_{M \ba b'}(\{b,x',y',x,y\})$.
But $a \in \cocl(\{b,x,y,u\})$, where $u \in \{x',y'\}$, a contradiction.

\textit{(Case 1.3)} The gadget for $\{b,b'\}$ is Type~II or Type~III, and $u = z'$. 
Since $\{u,x,y\}$ is a triad of $M \ba a,b$, and $\{u,u',x'\}$ is a triangle, we may assume without loss of generality that $x \in \{u',x'\}$.
In either case, $\cl(\{b,x,y,u\}) = \cl(\{b,x',y',u'\})$.
In $M \ba b'$, the set $\{b,x',y',u'\}$ is $3$-separating, with $x \in \{u',x'\}$.
As $\{y,u\} \subseteq \cl(\{b,x',y',u'\})$, due to the triangles $\{b,x,y\}$ and $\{x',u',u\}$, the set $\{b,x',y',u',y,u\}$ is $3$-separating in $M \ba b'$ and in $M \ba a,b'$.
By \cref{conn2,conn5}, it follows that $r^*_{M \ba b'}(\{b,x',y',u',y,u\}) = r^*_{M \ba a,b'}(\{b,x',y',u',y,u\}) = 5$, so $a \notin \cocl_{M \ba b'}(\{b,x',y',u',y,u\})$.
But $a \in \cocl(\{b,x,y,u\})$, where $x \in \{u',x'\}$, a contradiction.

\textit{(Case 1.4)} The gadget for $\{b,b'\}$ is Type~III, and $u = w'$.
Since $\{b,x,y\}$ is a triangle and $\{b,x',y',u'\}$ is a cocircuit in $M \ba b'$, it follows from orthogonality that $\{x,y\}$ meets $\{x',y',u'\}$.  Without loss of generality, $x \in \{x',y',u'\}$.

If $a \in \{x',z'\}$, then $u$ is in a triad of $M \ba a,b$ with $b'$, contradicting that $\{u,x,y\}$ is the unique triad of $M \ba a,b$ containing $u$.
So $a \notin \{x',z'\}$.
If $x=u'$, then since $\{u',x',z'\}$ is a triangle and $\{x,y,a\}$ is a triad of $M \ba u,b$, we have that $y \in \{x',z'\}$.
But then $\{x,y\}$ spans $\{b,u',x',y'\}$, so $r(\{u',x',y'\})=2$, a contradiction.
So $x \in \{x',y'\}$.

Now $r(\{x,y,x',y',b\})=2$, so $\{x,y\} = \{x',y'\}$ by \cref{easywin}.
Observe that $a \neq u'$, for otherwise $b' \in \cocl_{M \ba u',b}(\{x',y'\}) \subseteq \cocl_{M \ba a,b}(\{x,y,u\})$.
But then the triangle $\{u',x',z'\}$ and the triad $\{x,y,a\}$ of $M \ba b,u$ intersect in a single element, contradicting orthogonality.

\smallskip

Now we consider case 2: where $b$ blocks in the gadget and $b' \in \cl(\{x',y'\})$.
In particular, observe that $\{b',x',y',u'\}$ is a $3$-separating cocircuit in $M \ba b$.

\textit{(Case 2.1)} Suppose $u=u'$. 
Consider $M \ba b,u$, and observe that $\{b',x',y'\}$ is a $2$-separating triangle-triad.
If $a \notin \{x',y'\}$, then this is also a $2$-separating triangle-triad in $M \ba a,b,u$, so,  by \cref{minor3conn}, $b'$ is $N$-contractible in this matroid and in $M \ba a,b$, a contradiction.
So suppose $a \in \{x',y'\}$.
Then $b'$ is in a series pair in $M \ba a,b,u$, so $b'$ is $N$-contractible in this matroid and in $M \ba a,b$, a contradiction.

\textit{(Case 2.2)} Suppose $u \in \{x',y'\}$. First we let $u = x'$. 
If $a = y'$, then $M \ba a,b$ has a triad $\{u,b',u'\}$ that is not $\{u,x,y\}$, a contradiction.
Similarly, $a \neq u'$.
So $a \notin \{u,b',y',u'\}$.
Now $\{u,x,y\}$ is a triad and $\{u,b',y'\}$ is a triangle in $M \ba a,b$, so by orthogonality $y' \in \{x,y\}$.
Without loss of generality, $y' = y$.

Suppose $b$ does not fully block in the gadget for $\{b,b'\}$.
Then $\{b,b',u,y,u'\}$ is a rank-$3$ cocircuit in $M$.
In $M \ba a$, the set $\{x,b,u,y\}$ is $3$-separating, and $b'$ is in the closure of this set, due to the triangle $\{u,y,b'\}$, so $\lambda_{M \ba a}(\{x,b,u,y,b'\}) \le 2$.
Now $u' \in \cocl_{M \ba a}(\{b,b',u,y\})$ and $u' \in \cl(\{b,b',u,y\})$, so $\lambda_{M \ba a}(\{x,b,u,y,b',u'\}) \le 1$.
Since $M\ba a$ is $3$-connected, this implies that $|E(M\ba a)| \le 7$, a contradiction.
So we may assume that $b$ fully blocks in the gadget for $\{b,b'\}$.

Next we argue that $a$ fully blocks in the gadget for $\{a,b\}$.
Suppose not.
Then $a \in \cl(\{u,x,y\})$ and $\{a,b,u,x,y\}$ is a $3$-separating cocircuit in $M$.
Since $M \ba b,b',u'$ has an $N$-minor, and $\{x',y'\}$ is a series pair in this matroid, $M \ba b,u'/x'$ has an $N$-minor.
Similarly, as $\{y',b'\}$ is a parallel pair in the latter matroid, $M \ba b,u',y'$ has an $N$-minor.
Finally, as $\{x',b'\}$ is a series pair in this matroid, $M \ba b,u'/b'$ has an $N$-minor, $N_1$ say.
Let $Z_1=E(M)-\{a,b,u,x,y,b',u'\}$.
As $\{a,b,u,x,y,b'\}$ is $3$-separating in $M$, with $u'$ in the coguts, it follows that $(\{a,u,x,y\}, Z_1)$ is a $2$-separation in $M \ba b,u'/b'$.
By \cref{minor3conn}, either $|\{a,u,x,y\} \cap E(N_1)| \le 1$ or $|Z_1 \cap E(N_1)| \le 1$.
First, assume that $|\{a,u,x,y\} \cap E(N_1)| \le 1$.
Since $a \in \cl(\{u,x,y\})$, the element $a$ is $N$-deletable in $M \ba b,u'/b'$, by \cref{minor3conn}.
So $b'$ is $N$-contractible in $M \ba a,b$, a contradiction.

Now assume $|Z_1 \cap E(N_1)| \le 1$.
We claim that at most one element in $Z_1$ is not $N$-flexible in $M \ba b,u'/b'$.
We may assume that the $2$-separation $(\{a,u,x,y\},Z_1)$ is exact, for otherwise every element in $Z_1$ is $N$-flexible.
Suppose that there exists some $z_1 \in Z_1 \cap \cl_{M \ba b,u'/b'}(\{a,u,x,y\})$.
Note that $b' \notin \cl_M(\{a,x,z_1\})$, due to the cocircuit $\{b,b',u,y,u'\}$.
Then, since $r_{M \ba b,u'/b'}(\{a,u,x,y\})=2$, we have $z_1 \in \cl_M(\{a,x\})$.
So $\{a,x,z_1\}$ is a triangle in $M$.
It follows that $z_1$ is the unique element in $Z_1 \cap \cl_{M \ba b,u'/b'}(\{a,u,x,y\})$, by orthogonality with the cocircuit $\{a,b,u,x,y\}$.
So each element in $Z_1-z_1$ is $N$-contractible in $M \ba b,u'/b'$, by \cref{minor3conn}.
Note that $|E(N_1)| \le 4$, since $|Z_1 \cap E(N_1)| \le 1$ and $\{u,y\}$ is a parallel pair in $M \ba b,u'/b'$.
Since $N_1$ is non-binary, $N_1 \cong U_{2,4}$.
But $(M\ba b,u'/b')|\{a,u,x,z_1\} \cong U_{2,4}$, so each element in $Z-z_1$ is $N$-deletable, and hence $N$-flexible, in $M\ba b,u'/b'$, which satisfies the claim.
So we may assume that $Z_1 \cap \cl_{M \ba b,u'/b'}(\{a,u,x,y\}) = \emptyset$.
Next suppose that there are distinct elements $z_1,z_1' \in Z_1 \cap \cocl_{M \ba b,u'/b'}(\{a,u,x,y\})$.
Then $\{z_1,z_1'\}$ is a series pair in $M \ba b,u'$, so $\{z_1,z_1',u'\}$ is a triad in $M \ba b$ and $M \ba b,b'$.
But $\{u',u,y\}$ is the unique triad containing $u'$ in $M \ba b,b'$.
We deduce that $|Z_1 \cap \cocl_{M \ba b,u'/b'}(\{a,u,x,y\})| \le 1$.
By \cref{minor3conn}, at most one element in $Z_1$ is not $N$-flexible in $M \ba b,u'/b'$, as claimed.
Now, $M \ba b,u'$ has at least six $N$-flexible elements, since $|Z_1|-1 \ge 6$, contradicting that $M \ba b,b',u'$ has at most four $N$-flexible elements.
We deduce that $a$ fully blocks in the gadget for $\{a,b\}$.

Suppose that $x=u'$.
Observe that $\{x,b,y,u\}$ is $3$-separating in $M \ba a$, and $b'$ is in the closure of this set, due to the triangle $\{u,y,b'\}$.
If $x=u'$, then $b' \in \cocl_{M \ba a}(\{x,b,y,u\})$, so $\lambda_{M \ba a}(\{x,b,y,u,b'\}) \le 1$.
Since $M$ is $3$-connected, this implies $|E(M\ba a)| \le 6$, a contradiction.
So $x \neq u'$.

Now, if the gadget for $\{b,b'\}$ is Type~I, then $(x,y,u,u',a,b,b')$ is a mega-gadget.
Suppose the gadget for $\{b,b'\}$ is not Type~I.
In order to show that $(x,y,u,u',a,b,b')$ is a mega-gadget, it remains to show that $z'=a$ and the gadget for $\{b,b'\}$ is Type~II.
Since $\{a,b,x,y,u\}$ is a cocircuit and $\{u,u',z'\}$ is a triangle, $\{a,x\}$ meets $\{u',z'\}$ by orthogonality.
But if $x \in \{u',z'\}$, then $\cl(\{x,y\}) \subseteq \cl(\{y',u',z'\})$, so $b \in \cl(\{x',y',u'\})$, contradicting that $b$ fully blocks in the gadget for $\{b,b'\}$.
So $a \in \{u',z'\}$.
If $a = u'$, then $\{x',b',y'\} = \{u,b',y\}$ is a triangle-triad of $M \ba a,b$, a contradiction.
So $a=z'$.

Now suppose the gadget for $\{b,b'\}$ is Type~III.
Then $M \ba b,b',w'$ has an $N$-minor.  As $\{x',z'\} = \{u,a\}$ is a series pair in this matroid, $M \ba b,b',w'/u$ has an $N$-minor.
Now $\{a,u'\}$ is a parallel pair in this matroid, so $M \ba a,b,b',w'/u$ has an $N$-minor, $N_2$ say.
Let $Z_2 = E(M) - \{a,b,u,x,y,u',b',w'\}$.
Note that $\lambda_{M \ba a,b}(\{u,y,u',b'\})=2$, with $\{x,w'\} \subseteq \cocl_{M \ba a,b}(\{u,y,u',b'\})$, so $(\{u,y,u',b'\},x,Z_2)$ is a path of $2$-separations in $M \ba a,b,w'$.
As $|E(N_2) \cap \{u,y,u',b'\}| \le 2$ but $|E(N_2)| \ge 4$, we have $|E(N_2) \cap \{u,y,u',b'\}| \le 1$, by \cref{minor3conn}, so $|E(N_2) \cap (Z_2 \cup x)| \ge 3$.  By another application of \cref{minor3conn} for the $2$-separation $(\{u,y,u',b',x\},Z_2)$, we have $|E(N_2) \cap \{u,y,u',b',x\}| \le 1$.
Thus $M \ba a,b,w'/b'$ has an $N$-minor, by yet another application of \cref{minor3conn} and since $b' \in \cocl_{M \ba a,b,w'}(\{u,y,u'\})$.
In particular, $b'$ is $N$-contractible in $M \ba a,b$, a contradiction.
Thus the gadget for $\{b,b'\}$ is the Type~II gadget $(u,y,u',a)$.
So $(x,y,u,u',a,b,b')$ is a mega-gadget.

The argument when $u=y'$ is essentially the same, with the roles of $x'$ and $y'$ swapped.  Note that in the subcase where the gadget for $\{b,b'\}$ is not Type~I, a similar argument shows that the gadget for $\{b,b'\}$ is the Type~II gadget $(y,u,u',a)$, so again $(x,y,u,u',a,b,b')$ is a mega-gadget.

\textit{(Case 2.3)} Suppose the gadget for $\{b,b'\}$ is Type~II or Type~III, and $u = z'$. 
Since $M \ba a,b,b' / u$ has an $N$-minor, and $\{u',x'\}$ is a parallel pair in this matroid, $M \ba a,b,b',u'/u$ has an $N$-minor, $N_1$ say.
The matroid $M \ba b / u$ has $\{b',x',y',u'\}$ as a rank-$2$ cocircuit, where $\{u',x'\}$ is a parallel pair.
So $\{b',x',y',u'\}$ is a $2$-separation in $M \ba b/u$. 
If $a \notin \{b',x',y',u'\}$, then $\{b',x',y',u'\}$ is also a $2$-separation in $M \ba a,b/u$.
As $|\{b',x',y',u'\} \cap E(N_1)| \le 2$ and $|E(N_1)| \ge 4$, \cref{minor3conn} implies that $|\{b',x',y',u'\} \cap E(N_1)| \le 1$ and $b'$ is $N$-flexible in $M \ba a,b/u$.
In particular, $b'$ is $N$-contractible in $M \ba a,b$, a contradiction.
So we may assume that $a \in \{x',y',u'\}$.
Now, in a similar fashion, $\{b',x',y',u'\}-a$ is a $2$-separation in $M \ba a,b/u$, and, by \cref{minor3conn}, $|\{b',x',y',u'\}-a \cap E(N_1)| \le 1$ and $b'$ is $N$-flexible in $M \ba a,b/u$.
In particular, $b'$ is $N$-contractible in $M \ba a,b$, a contradiction.

\textit{(Case 2.4)} Suppose the gadget for $\{b,b'\}$ is Type~III, and $u = w'$. 
Since $M \ba a,b,b',u$ has an $N$-minor, and $\{x',z'\}$ is a series pair in this matroid, $M \ba a,b,b',u/x'$ has an $N$-minor.
As $\{u',z'\}$ is a parallel pair in the latter matroid, $M \ba \{a,b,b',u,z'\}/x'$ has an $N$-minor, $N_3$ say.
Let $Z_3 = E(M \ba b,u) - \{x',y',u',z',b'\}$.
The matroid $M \ba a,b,u$ has a $2$-separation that is either $(\{x',y',u',z',b'\}-a,Z_3)$, if $a \in \{x',y',u',z'\}$; or $(\{x',y',u',z',b'\},Z_3-a)$, if $a \in Z_3$.
Note that $|(\{x',y',u',z',b'\}-a) \cap E(N_3)| \le 2$, so, as $|E(N_3)| \ge 4$, \cref{minor3conn} implies that $|(\{x',y',u',z',b'\}-a) \cap E(N_3)| \le 1$.
Now $b' \in \cocl_{M \ba a,b,u}(\{u',x',y'\}-a)$, so $b' \notin \cl_{M \ba a,b,u}(Z_3-a)$, implying $b'$ is $N$-contractible in $M \ba a,b \ba u$ and hence in $M \ba a,b$, a contradiction.

From this case analysis, we deduce that either $u$ is not $N$-flexible in $M \ba a,b,b'$, or there is a mega-gadget, as required.
\end{subproof}

\begin{claim}
  \label{megagadgetcase}
  If $M \ba b,b'$ is $3$-connected and $(x,y,u,u',a,b,b')$ is a mega-gadget, then $u$ is not $N$-flexible in $M \ba a,b'$.
\end{claim}
\begin{subproof}
  Assume that $M \ba b,b'$ is $3$-connected and $(x,y,u,u',a,b,b')$ is a mega-gadget.
  By the definition of a mega-gadget, the gadget for $\{b,b'\}$ is either $(\{u,y\},u')$ if Type~I, or $(u,y,u',a)$ or $(y,u,u',a)$ if Type~II.
  In any case, $b$ fully blocks and $\{b',u,y\}$ is a triangle.
  Note that $\{b,b',u,y,u'\}$ is a cocircuit of $M$, so $u' \in \cocl(\{b,b',u,y\})$.
  Let $Z = E(M\ba a)-\{x,b,y,u,u',b'\}$.

  We claim that $M \ba a,b'$ is $3$-connected.
  By \cref{conn2}, $M \ba a,b'$ is $3$-connected up to series classes and if $b'$ is in a triad of $M \ba a$, then this triad avoids $\{x,y\}$.
  Suppose $b'$ is in a triad~$T^*$ of $M \ba a$.
  Since $M \ba a,b$ is $3$-connected, $b \notin T^*$ and $T^*$ is a triad of $M \ba a,b$.
  By orthogonality with the triangle $\{b',u,y\}$, and the foregoing, $u \in T^*$.
  But then $T^*$ is a triad of $M \ba a,b$ containing $u$ and distinct from $\{u,x,y\}$, a contradiction.
  So $M \ba a,b'$ is $3$-connected.

  Recall that $b'$ is $N$-deletable in $M \ba a,b$, so $M \ba a,b'$ has an $N$-minor.
  If $M \ba a,b'$ is $N$-fragile, then $u$ is not $N$-flexible in $M \ba a,b'$, as required.
  So we may assume that $M \ba a,b'$ is not $N$-fragile.
  Then there is a gadget for $\{a,b'\}$.
  Let this gadget be $(\{x'',y''\},u'')$, $(x'',y'',u'',z'')$, or $(x'',y'',u'',z'',w'')$ if Type~I, Type~II, or Type~III respectively.

  There are two cases to consider:
  \begin{enumerate}[label=\rm(\Roman*)]
    \item $a$ blocks in the gadget for $\{a,b'\}$, and $b' \in \cl(\{x'',y''\})$; or
    \item $b'$ blocks in the gadget for $\{a,b'\}$, and $a \in \cl(\{x'',y''\})$.
  \end{enumerate}

  First we consider case~(I).
  We consider two subcases depending on whether or not $\{x'',y''\} = \{u,y\}$.

  Suppose $\{x'',y''\} = \{u,y\}$.
  Assume that the gadget for $\{a,b'\}$ is Type~II or Type~III, and $(x'',y'') = (y,u)$.
  Then, by orthogonality between the triad $\{x'',y'',u''\}$ of $M \ba a,b'$, and the triangle $\{x,b,y\}$, we have $u'' \in \{x,b\}$.  But then the triad $\{x'',y'',u''\}$ of $M \ba a,b'$ is properly contained in the $4$-element cocircuit $\{x,b,y,u\}$, a contradiction.
  We deduce that either the gadget for $\{a,b'\}$ is Type~I, or $(x'',y'') = (u,y)$.
  So we may assume (in either case) that $y=y''$ and $u=x''$.
  By orthogonality between the triangle $\{b,x,y\}$ and the triad $\{u'',u,y\}$ of $M \ba a,b'$, we have $u'' \in \{b,x\}$.
  The matroid $M \ba a,u''$ is not $3$-connected, since it has a triangle-triad $\{x'',y'',b'\}$.
  Since $M \ba a,b$ is $3$-connected, but $M \ba a,u''$ is not, $u'' = x$.
  Now $\{x'',b',y'',u''\} = \{u,b',y,x\}$ is a cocircuit in $M \ba a$.
  Since $\{b,u,x,y\}$ is also a cocircuit in $M \ba a$, the set $\{b,b',u,x\}$ contains a cocircuit, by cocircuit elimination.
  Since $M \ba a,b$ is $3$-connected, this cocircuit is either $\{b',u,x\}$ or $\{b,b',u,x\}$.
  But then $\{b',u,x\}$ is a triad of $M \ba a,b$ that contains $u$ and is distinct from $\{u,x,y\}$, a contradiction.

  Now $\{x'',y''\} \neq \{u,y\}$.
  By \cref{easywin} relative to the gadget for $\{b,b'\}$, we have $\cl(\{b',u,y\}) = \{b',u,y\}$, so $\{x'',y''\} \cap \{u,y\} = \emptyset$.
  By orthogonality between the triangle $\{b',x'',y''\}$ and the cocircuit $\{b,b',u',u,y\}$, we have $u' \in \{x'',y''\}$ or $b \in \{x'',y''\}$.
  But in the latter case $\{x'',y''\} = \{b,x\}$, by orthogonality between the triangle $\{x'',b',y''\}$ and the cocircuit $\{b,u,x,y\}$ of $M \ba a$, in which case $\{b,x,y\}$ is not closed, contradicting \cref{easywin}.
  So $u' \in \{x'',y''\}$.

  Now $\{b',x'',y''\} = \{b',u',q\}$ for some element $q$ with $q \notin \{u,y\}$.
  We have $\lambda_{M \ba a}(\{x,b,y,u,b',u'\}) = 2$ and $q \in \cl(\{x,b,y,u,b',u'\})$.
  Now $(\{x,b,y,u,b',u'\},q,Z-q)$ is a path of $3$-separations in $M \ba a$, with $q$ in the guts, since $|E(M \ba a)| \ge 9$.
  As $\{b',x'',y'',u''\}=\{b',u',q,u''\}$ is a cocircuit of $M \ba a$, we have $u'' \in Z-q$, for otherwise $q \in \cocl_{M \ba a}(\{x,b,y,u,b',u'\})$.
  But orthogonality with the triangle $\{b',u,y\}$ implies that $u'' \in \{u,y\}$, a contradiction.

  Now consider case~(II).
  Recall that $Z = E(M\ba a)-\{x,b,y,u,u',b'\}$.
  If $\{x'',y''\} \subseteq Z \cup \{u',b'\}$, then $a \in \cl(Z \cup \{u',b'\})$, due to the triangle $\{x'',a,y''\}$.
  But then $a \notin \cocl(\{x,b,y,u\})$, a contradiction.
  So $\{x'',y''\}$ meets $\{x,b,y,u\}$.
  Moreover, as $a$ fully blocks the gadget $(\{x,y\},u)$ for $\{a,b\}$, $|\{x'',y''\} \cap \{x,b,y,u\}| = 1$.
  Note also that if $u'' = b$, then $\{x'',a,y''\}$ is a triangle-triad in $M \ba b,b'$, contradicting that $M \ba b,b'$ is $3$-connected.
  So $u'' \neq b$.

  First, suppose that $x \in \{x'',y''\}$.
  In what follows, we assume that $x=x''$; if $x=y''$, then essentially the same argument applies but with the roles of $x''$ and $y''$ swapped (in particular, note that the argument makes no assumption about whether the gadget for $\{a,b'\}$ is Type~I, Type~II, or Type~III).
  By orthogonality between the cocircuit $\{a,b',x'',y'',u''\}$ and the triangles $\{x'',b,y\}$ and $\{b',u,y\}$, the set $\{y'',u''\}$ meets both $\{b,y\}$ and $\{u,y\}$.
  So either $y \in \{y'',u''\}$, or $\{y'',u''\} = \{b,u\}$.
  But in the latter case $b=y''$ and $u=u''$, since $u'' \neq b$.
  Then, as $\{a,b',x'',y'',u''\}$ is a cocircuit in $M$, the set $\{b',x'',u\}$ is a triad in $M \ba a,b$ that is distinct from $\{u,x,y\}$, a contradiction.
  So $y \in \{y'',u''\}$.
  If $y=y''$, then the triangle $\{x,a,y\}$ intersects the cocircuit $\{b,b',u,y,u'\}$ in a single element, contradicting orthogonality.
  So $y = u''$. 

  Due to the cocircuit $\{a,b',u'',x'',y''\}$, we have $y'' \in \cocl_{M \ba a}(\{b',u'',x''\}) \subseteq \cocl_{M \ba a}(\{x,b,y,u,b'\})$.
  As $y'' \in \cl(\{a,x\}) \subseteq \cl(\{x,b,y,u,b',a\})$, and $\lambda_{M \ba a}(\{x,b,y,u,b'\})=2$, it follows that $\lambda_M(\{x,b,y,u,b',a,y''\})=2$.

  Let $Z' = Z-y''$.
  As $u' \in \cocl(\{b,y,u,b'\})$, we have that $(\{x,b,y,u,b',a,y''\},u',Z')$ is a path of $3$-separations in $M$.
  Hence $(\{x,y,u,a,y''\},Z')$ is a $2$-separation in $M \ba b,b',u'$.
  Moreover, $u'$ is $N$-deletable in $M \ba b,b'$. 
  By \cref{minor3conn}, the $N$-minor of $M \ba b,b',u'$, say $N_1$, has either $|\{x,y,u,a,y''\} \cap E(N_1)| \le 1$ or $|Z' \cap E(N_1)| \le 1$.

  Suppose $|\{x,y,u,a,y''\} \cap E(N_1)| \le 1$.
  Then $(\{x,y,u,b',a,y''\},Z')$ is a $2$-separation in $M \ba b,u'$ with $|\{x,y,u,b',a,y''\} \cap E(N_1)| \le 1$, by \cref{minor3conn} and since $|E(N)| \ge 4$.
  Since $u \in \cl(\{y,b'\})$, we have $u \notin \cocl_{M \ba b,u'}(Z')$, implying $M \ba b,u',u$ has an $N$-minor, by \cref{minor3conn}.
  As $\{b',y\}$ is a series pair in this matroid, $M \ba b,u'/b'$ has an $N$-minor, $N_2$ say.
  Now $(\{x,y,u,a,y''\}, Z')$ is a $2$-separation in $M \ba b,u'/b'$, with $|\{x,y,u,a,y''\} \cap E(N_2)| \le 1$.
  As $a \in \cl(\{x,y''\})$, we have $a \notin \cocl_{M \ba b,u'/b'}(Z')$, so $a$ is $N$-deletable in $M \ba b,u'/b'$ by another application of \cref{minor3conn}.
  But then $b'$ is $N$-contractible in $M \ba a,b$, a contradiction.

  Now suppose $|Z' \cap E(N_1)| \le 1$. 
  By \cref{gutsandcoguts}, $|Z' \cap \cl_{M \ba b,b',u'}(\{x,y,u,a,y''\})| = |Z' \cap \cl_M(\{x,b,y,u,b',a,y''\})| \le 1$, so at most one element of $Z'$ is not $N$-contractible in $M \ba b,b',u'$, by \cref{minor3conn}.
  Suppose there are distinct elements $q,q' \in Z' \cap \cocl_{M \ba b,b',u'}(\{x,y,u,a,y''\})$.
  Then $\{q,q',u'\}$ is in the coguts of the $3$-separation $(E(M)-Z',Z')$, so $\{q,q',u'\}$ is a triad of $M$, a contradiction.
  So $|Z' \cap \cocl_{M \ba b,b',u'}(\{x,y,u,a,y''\})| \le 1$, and at most one element of $Z'$ is not $N$-deletable in $M \ba b,b',u'$, by \cref{minor3conn}.
  Thus, at most two elements of $Z'$ are not $N$-flexible in $M \ba b,b',u'$, so $Z'$ contains at least four $N$-flexible elements of $M \ba b,b',u'$, since $|Z'| = |E(M)| - 8 \ge 6$.
  But this contradicts that $M \ba b,b',u'$ has at most three $N$-flexible elements (since $u'$ is in the gadget for $M \ba b,b'$, which is Type~I or Type~II).

  Now we may assume that $x \notin \{x'',y''\}$, so $\{x'',y''\}$ meets $\{b,u,y\}$.
  Then $\{x'',y''\} \subseteq \{b,u,y,u'\}$, by orthogonality between the triangle $\{x'',a,y''\}$ and the cocircuit $\{b,b',u,y,u'\}$.
  Moreover, recalling that $|\{x'',y''\} \cap \{x,b,y,u\}| = 1$,
  we have $u' \in \{x'',y''\}$.
  We first assume that $x'' \in \{b,u,y\}$ and $y'' = u'$.
  In making this assumption, we lose no generality when the gadget for $\{a,b'\}$ is Type~I; we later return to the case that $y'' \in \{b,u,y\}$ and $x'' = u'$ and the gadget for $\{a,b'\}$ is Type~II or Type~III.

  Firstly, suppose that $x'' = u$.
  Then the cocircuit $\{a,b',u,u',u''\}$ intersects the triangle $\{b,x,y\}$ in at most one element, $u''$.
  By orthogonality, $u'' \notin \{b,x,y\}$.
  Now $\lambda_{M \ba a}(\{b,u,x,y,b'\})=2$, with $u' \in \cl(\{u,a\})$ and $u' \in \cocl(\{b,b',u,y\})$, due to the triangle $\{x'',a,y''\}$ and cocircuit $\{b,b',u',x',y'\}$ respectively.
  It follows that $\lambda_M(\{a,b,u,x,y,u',b'\}) = 2$.
  Moreover, due to the cocircuit $\{a,b',u,u',u''\}$, we have $u'' \in \cocl(\{a,b,u,x,y,u',b'\})$.
  Let $Z'' = Z - u''$.
  Hence $(\{b,u,x,y,u'\},Z'')$ is a $2$-separation in $M \ba a,b',u''$.
  This matroid has an $N$-minor, $N_3$ say.
  So, by \cref{minor3conn}, either $|\{b,u,x,y,u'\} \cap E(N_3)| \le 1$ or $|Z'' \cap E(N_3)| \le 1$.

  Assume $|\{b,u,x,y,u'\} \cap E(N_3)| \le 1$.
  Then $(\{b,u,x,y,u',b'\},Z'')$ is a $2$-separation in $M \ba a,u''$ with $|\{b,u,x,y,u',b'\} \cap E(N_3)| \le 1$, by \cref{minor3conn} and since $|E(N)| \ge 4$.
  Now $b \in \cl(\{x,y\})$, so $b \notin \cocl_{M \ba a,u''}(Z'')$, implying that $M \ba a,b,u''$ has an $N$-minor, say $N_4$, by \cref{minor3conn}.
  Furthermore, $(\{u,x,y,u',b'\},Z'')$ is a $2$-separation in this matroid, with $|\{u,x,y,u',b'\} \cap E(N_4)| \le 1$ and $b' \in \cocl_{M \ba a,b,u''}(\{u',u,y\})$, so $b' \notin \cl_{M \ba a,b,u''}(Z'')$, implying $M \ba a,b,u''/b'$ has an $N$-minor, by \cref{minor3conn}.
  So $b'$ is $N$-contractible in $M \ba a,b$, a contradiction.

  Now assume $|Z'' \cap E(N_3)| \le 1$.
  Recall that $\lambda_M(\{a,b,u,x,y,u',b'\}) = 2$ and $u'' \in \cocl(\{a,b,u,x,y,u',b'\})$.
  By \cref{gutsandcoguts}, $|Z \cap \cl_M(\{a,b,u,x,y,u',b'\})| \le 1$.
  But $|Z'' \cap \cl_{M \ba a,b',u''}(\{b,u,x,y,u'\})| \le |Z \cap \cl_M(\{a,b,u,x,y,u',b'\})|$, so at most one element of $Z''$ is not $N$-contractible in $M \ba a,b',u''$, by \cref{minor3conn}.
  Suppose there are distinct elements $q,q' \in Z'' \cap \cocl_{M \ba a,b',u''}(\{b,u,x,y,u'\})$.
  Then $\{q,q',u''\}$ is in the coguts of the $3$-separation $(E(M)-Z'',Z'')$, so $\{q,q',u''\}$ is a triad of $M$, a contradiction.
  So $|Z'' \cap \cocl_{M \ba a,b',u''}(\{b,u,x,y,u'\})| \le 1$, and at most one element of $Z''$ is not $N$-deletable in $M \ba a,b',u''$, by \cref{minor3conn}.
  Thus, at most two elements of $Z''$ are not $N$-flexible in $M \ba a,b',u''$, so $Z''$ contains at least five $N$-flexible elements of $M \ba a,b',u''$, since $|Z''| = |E(M)| - 8 \ge 7$.
  But this contradicts that $M \ba a,b',u''$ has at most four $N$-flexible elements, since $u''$ is in the gadget for $M \ba a,b'$.

  Secondly, suppose that $x'' = b$. 
  Then, by orthogonality between the cocircuit $\{a,b',b,u',u''\}$ and the triangles $\{x,b,y\}$ and $\{u,b',y\}$, we have $u''=y$.
  We may assume that $u$ is $N$-flexible in $M \ba a,b'$, for otherwise \cref{megagadgetcase} holds.
  So $u$ is in the gadget for $\{a,b'\}$.
  But $u \notin \{x'',y'',u''\}$, as $\{x'',y'',u''\}=\{b,u',y\}$, so the gadget for $\{a,b'\}$ is Type~II or Type~III.
  By \cref{easywin}, $\cl(\{b,x,y\}) = \cl(\{x'',u''\}) = \{x'',u'',x\}$.
  So the gadget for $\{a,b'\}$ is Type~III, $z''=x$, and $w''=u$.
  But then $\{x'',z'',w''\}=\{b,x,u\}$ is a triad in $M \ba a,b'$, contradicting that $\{b,x,y,u\}$ is a cocircuit in $M \ba a,b'$, since $(\{x,y\},u)$ is the gadget for $\{a,b\}$ and $b' \notin \cocl_{M \ba a}(\{b,x,y,u\})$.

  Thirdly, suppose that $x'' = y$.
  We use a similar argument as in the case that $x''=b$.
  Recall that $u'' \neq b$.
  By orthogonality between the cocircuit $\{a,b',y,u',u''\}$ and the triangle $\{x,b,y\}$, we have $u'' =x$.
  We may assume that $u$ is $N$-flexible in $M \ba a,b'$, for otherwise \cref{megagadgetcase} holds.
  So $u$ is in the gadget for $\{a,b'\}$.
  But $u \notin \{x'',y'',u''\}$, as $\{x'',y'',u''\}=\{y,u',x\}$, so the gadget for $\{a,b'\}$ is Type~II or Type~III.
  By \cref{easywin}, $\cl(\{b,x,y\}) = \cl(\{x'',u''\}) = \{x'',u'',b\}$.
  So the gadget for $\{a,b'\}$ is Type~III, $z''=b$, and $w''=u$.
  But then $\{x'',z'',w''\}=\{y,b,u\}$ is a triad in $M \ba a,b'$, contradicting that $\{b,x,y,u\}$ is a cocircuit in $M \ba a,b'$, since $(\{x,y\},u)$ is the gadget for $\{a,b\}$ and $b' \notin \cocl_{M \ba a}(\{b,x,y,u\})$.

  It remains to consider the cases where $y'' \in \{b,u,y\}$ and $x'' = u'$ and the gadget for $\{a,b'\}$ is Type~II or Type~III.
  When $y''=u$, we can argue as in the case that $x''=u$ and $y'' = u'$ (as this argument makes no assumptions on whether the gadget for $\{a,b'\}$ is Type~I, Type~II, or Type~III). 
  So we need only consider the cases where $y'' =b$ or $y'' = y$.
  First suppose $y'' = b$.
  By orthogonality between the cocircuit $\{a,b',b,u',u''\}$ and the triangles $\{x,b,y\}$ and $\{u,b',y\}$, it follows that $u''=y$.
  As $u$ is $N$-flexible in $M \ba a,b'$, the element $u$ is in the gadget for $\{a,b'\}$. 
  The triangle $\{x'',u'',z''\} = \{u',y,z''\}$ intersects the cocircuit $\{a,b,u,x,y\}$ in at least one element so, by orthogonality, $z'' \in \{u,x\}$.
  But neither $\{u',y,u\}$ nor $\{u',y,x\}$ is a triangle, by \cref{easywin}, a contradiction.

  Finally, suppose $y'' = y$.
  By orthogonality between the cocircuit $\{a,b',y,u',u''\}$ and the triangle $\{x,b,y\}$, and since $u'' \neq b$, we have $u'' =x$.
  Recall that $u$ is in the gadget for $\{a,b'\}$, which is Type~II or Type~III.
  The triangle $\{x'',u'',z''\} = \{u',x,z''\}$ intersects the cocircuit $\{a,b,u,x,y\}$ in at least one element so, by orthogonality, $z'' \in \{b,u\}$.
  But $\{u',x,b\}$ is not a triangle, by \cref{easywin}, and neither is $\{u',x,u\}$, for otherwise $b$ does not fully block in the gadget for $\{b,b'\}$.
  This contradiction completes the proof.
\end{subproof}

Let $b' \in B-\{x,y\}$, and suppose that $M \ba b,b'$ is $3$-connected.
By \cref{caseanalysis,megagadgetcase}, $u$ is not $N$-flexible in $M \ba a,b,b'$.
By \cref{setup}, we deduce that $b'$ is $N$-essential in one of $M \ba a,b,u$ and $M \ba a,b / u$.
So the theorem holds for any $b'$ such that $M \ba b,b'$ is $3$-connected.
It remains to consider the case where $M \ba b,b'$ is not $3$-connected.
In this case, either $M \ba b,b'$ has a triangle-triad, or one or two series pairs, by \cref{conn3}.
We next address the subcase where $M \ba b,b'$ has a triangle-triad.

\begin{claim}
  \label{triangletriadcase}
  If $M \ba b,b'$ has a triangle-triad $\{x,s,a\}$, then $u$ is not $N$-flexible in $M \ba a,b,b'$.
\end{claim}
\begin{subproof}
  Suppose $M \ba b,b'$ has a triangle-triad $\{x,s,a\}$.
  Then, by \cref{conn3}, $M \ba a,b'$ is $3$-connected.
  We assume, towards a contradiction, that $u$ is $N$-flexible in $M \ba a,b,b'$.
  In particular, $M \ba a,b'$ is not $N$-fragile.

  First, suppose that $a$ does not fully block the gadget $(\{x,y\},u\}$ for $\{a,b\}$.
  Then $\cl(\{a,b,x\})$ is $3$-separating in $M$, and $u \in \cl(\{a,b,x\})$, by \cref{easywin}.
  Since $M \ba b'$ is $3$-connected, by \cref{conn5}, $b$ fully blocks the triangle-triad $\{x,s,a\}$ of $M \ba b,b'$.
  So $\{x,s,a,b\}$ is a $3$-separating cocircuit in $M \ba b'$.
  Let $Q = \cl_{M \ba b'}(\{x,s,a,b\}) - y$, and observe that $Q$ is also $3$-separating in $M \ba b'$.
  Note that $u \in Q$, since $u \in \cl(\{a,b,x\})$.
  Now $y \in \cocl_{M \ba b'}(Q)$, since $\{a,b,x,y,u\}$ is a cocircuit, and $y \in \cl_{M \ba b'}(Q)$, due to the triangle $\{b,x,y\}$.
  Thus $\lambda(Q \cup y) < 2$, implying 
  $Q \cup y = E(M \ba b')$, as $Q \cup y$ is closed and $M \ba b'$ is $3$-connected.
  Let $L=E(M \ba b')-\{a,b,s,x\}$.
  Then $r(L)=2$, and $|L| \ge 9$, so $L$ contains a triangle that intersects $\{a,b,u,x,y\}$ in a single element, contradicting orthogonality.

  Now suppose that $a$ fully blocks the gadget $(\{x,y\},u)$ for $\{a,b\}$.
  Recall that $M \ba a,b'$ is not $N$-fragile and consider the gadget for $\{a,b'\}$.
  Let this gadget be $(\{x',y'\},u')$ if Type~I, $(x',y',u',z')$ if Type~II, or $(x',y',u',z',w')$ if Type~III.
  There are two cases: 
  \begin{enumerate}
    \item[(i)] $b'$ blocks in the gadget, and $a \in \cl(\{x',y'\})$, or
    \item[(ii)] $a$ blocks in the gadget, and $b' \in \cl(\{x',y'\})$.
  \end{enumerate}
  In case~(i), $b'$ may or may not fully block $\{x',y',u',a\}$, whereas in case~(ii), $a$ may or may not fully block $\{x',y',u',b'\}$.

  We first handle, simultaneously, the non-fully-blocking subcases.
  In these cases, $\{a,b',x',y',u'\}$ is $3$-separating, so this set is a rank-$3$ cocircuit.
  Let $Q = \cl(\{a,b',x',y',u'\})$, so $Q$ is $3$-separating.
  Now, in $M \ba b'$, the set $\{x,s,a\}$ is a triangle, and $\{a,x',y',u'\}$ is a cocircuit, so by orthogonality $\{x,s,a\} \subseteq Q$. 
  Observe that since $\{a,b',x',y',u'\}$ is a cocircuit, $b'$ is not in the guts of the $3$-separation $(Q,E(M)-Q)$.
  Suppose $x$ is in the guts of the $3$-separation $(Q, E(M)-Q)$.
  Then, as $M \ba a,b/x$ has an $N$-minor, by \cref{flexigadgets}, and this matroid 
  is a $2$-sum, each element of $Q$ that is not at the basepoint of the $2$-sum is $N$-contractible.
  In particular, $b'$ is $N$-contractible in $M \ba a,b$, a contradiction.
  So we may assume that $x$ is not in the guts of the $3$-separation $(Q, E(M)-Q)$.
  It now follows, by orthogonality, that the triangle $\{b,x,y\}$ is contained in $Q$.
  Hence $\{a,b,x,y,s,b'\} \subseteq Q$.
  In particular, $b'$ does not fully block $\{x,s,a,b\}$, which is a $3$-separating cocircuit of $M \ba b'$.
  But $y$ is in the guts of this $3$-separation, due to the triangle $\{b,x,y\}$, and $M \ba a,b/y$ has an $N$-minor.
  Again, it follows that $b'$ is $N$-contractible, a contradiction.

  We may now assume, for case~(i) and case~(ii), that $b'$ or $a$, respectively, fully blocks in the gadget for $\{a,b'\}$.
  Suppose we are in case~(i). 
  Since $\{x,s,a\}$ is a triangle and $\{a,u',x',y'\}$ is a cocircuit of $M \ba b'$, by orthogonality $\{s,x\}$ meets $\{u',x',y'\}$.
  
  Suppose $x \notin \{u',x',y'\}$.
  By \cref{easywin}, the triangle $\{x',y',a\}$ is closed.
  Thus $s \notin \{x',y'\}$, for otherwise $\{x',y',a\} = \{x,s,a\}$.
  So $s = u'$.
  But now $\{x,s,a\}$ is a triad in $M \ba b,b'$ that intersects the triangle $\{x',y',a\}$ in a single element, contradicting orthogonality.

  So we may assume that $x \in \{u',x',y'\}$.
  Again by orthogonality, this time due to the triangle $\{b,x,y\}$, we have that $\{b,y\}$ meets $\{u',x',y'\}$.
  So $\{u',x',y'\}$ contains $x$ and either $b$ or $y$.

  Suppose the other element in this set is $u$.
  By \cref{easywin}, $\{x',y',a\}$ is closed, so $\{x',y'\}\nsubseteq \{b,x,y\}$, implying $u \in \{x',y'\}$.
  Without loss of generality, let $u=y'$ and $\{x',u'\} \subseteq \{x,b,y\}$.
  But now, in $M \ba a,b$, either $\{b',x,y,u\}$ or $\{b',x,u\}$ is a cocircuit, contradicting that $\{u,x,y\}$ is the unique triad containing $u$ in $M \ba a,b$.
  So $u \notin \{u',x',y'\}$.
  By the assumption that $u$ is $N$-flexible in $M \ba a,b'$,
  we deduce that the gadget for $\{a,b'\}$ is not Type~I.

  Next, we claim that $u \notin \cl(\{u',x',y'\})$, which will imply that the gadget for $\{a,b'\}$ is not Type~II.
  Consider the set $P=\{a,u',x',y'\} \cup \{b,x,y\}$.
  Recall that $x \in \{u',x',y'\}$, and, by orthogonality, $\{b,y\}$ meets $\{u',x',y'\}$, so $|P|=5$.
  Observe also that $P$ is $3$-separating in $M \ba b'$.
  Since $\{a,b,x,y,u\}$ is a cocircuit, $u \in \cocl_{M \ba b'}(P)$.
  If $u \in \cl(P)=\cl(\{u',x',y'\})$, then $\lambda_{M \ba b'}(P \cup u) < 2$, implying $|E(M\ba b')-(P \cup u)| \le 1$.
  But then $|E(M)| \le |P \cup u| + 2 \le 8$, a contradiction.
  So the gadget for $\{a,b'\}$ is not Type~II.

  Thus the gadget for $\{a,b'\}$ is Type~III and $u=w'$.
  Recall that $(y',u',x',z',w')$ is a fan ordering of $M \ba a,b'$. 
  Since the $5$-element fan is maximal 
  and contains two elements of the triangle $\{b,x,y\}$, it must contain all three elements; that is, $\{u',x',z'\} = \{b,x,y\}$.
  Then $\{u',x',z',w'\}$ is a cocircuit in $M \ba a,b'$, and $u'$ is in the closure and coclosure of the triad $\{x',z',w'\}$ of $M \ba a,b'$, so $\lambda_{M \ba a,b'}(\{x',z',w',u'\}) = 1$, a contradiction.
  This completes case (i).

  Suppose we are in case (ii).
  Recall that $\{a,s,x\}$ is a triangle-triad of $M \ba b,b'$, so, in particular, it is a triangle of $M \ba b'$.
  Since $\{a,x',y',u'\}$ is a cocircuit of $M \ba b'$, we deduce, by orthogonality, that $\{s,x\} \cap \{x',y',u'\} \neq \emptyset$.

  Consider the case where $x \in \{x',y',u'\}$.
  If $s \in \{x',y',u'\}$, then $a \in \cl(\{x',y',u'\})$, contradicting that $a$ fully blocks in the gadget for $\{a,b'\}$.
  So $s \notin \{x',y',u',a,b'\}$.
  Now $\{x,b,y\}$ is a triangle that meets the triad $\{u',x',y'\}$ of $M \ba a,b'$ in the element $x$, so, by orthogonality and since $M \ba a,b'$ is $3$-connected, the triangle and triad intersect in precisely two elements.
  Thus either $y$ or $b$ is in the guts of the $3$-separating set $\{b',x',y',u'\}$ of $M \ba a$.
  In the former case, where $y$ is in the guts, we have $b \in \{u',x',y'\}$.
  As $M \ba a,b / y$ has an $N$-minor, and $\{b',u',x',y'\}-b$ is a triangle-triad in this matroid, it follows from \cref{minor3conn} that $b'$ is $N$-contractible, a contradiction.
  In the latter case, where $b$ is in the guts, we have $\{x,y\} \subseteq \{u',x',y'\}$.
  By \cref{easywin}, 
  one of $x$ and $y$ is $u'$, whereas the other is in $\{x',y'\}$.
  In $M \ba a,b$, the set $\{x,y,u\}$ is a triad, $\{b',x',y'\}$ is a triangle, and these two sets meet.
  By orthogonality, and since $b' \notin \{x,y,u\}$, we have $u \in \{x',y'\}$.
  So $\{u,x,y\} = \{u',x',y'\}$, and $\{b',u',x',y'\} = \{b',x,y,u\}$ is a cocircuit of $M \ba a,b$, a contradiction.

  We may now assume we are in the case where $s \in \{x',y',u'\}$ and $x \notin \{x',y',u'\}$.
  In $M \ba a,b'$, the unique triad containing $u'$ is $\{u',x',y'\}$.
  Since $\{a,s,x\}$ is a triangle-triad of $M \ba b,b'$ that is fully blocked by $b$, since $M \ba b'$ is $3$-connected, the set $\{a,b,s,x\}$ is a cocircuit of $M \ba b'$.
  So $\{s,b,x\}$ is a triad of $M \ba a,b'$.
  Suppose $s=u'$.
  Now $\{b,x\} = \{x',y'\}$, but then $b' \in \cl(\{b,x,y\})-\{b,x,y\}$, contradicting \cref{easywin}.
  So $s \in \{x',y'\}$.

  Now $x \notin \cl(\{x',y',u'\})$, since $\{x,s,a\}$ is a triangle and $a$ fully blocks in the gadget for $\{a,b'\}$.
  Moreover, as $\{x,b,y\}$ is a triangle, and $\{a,b',u',x',y'\}$ is a cocircuit that avoids $x$, this also implies, by orthogonality, that $\{b,y\} \cap \{u',x',y'\} = \emptyset$.

  Suppose that $u \in \{u',x',y'\}$.
  Note that $\{x',y'\} \neq \{s,u\}$, by orthogonality between the triangle $\{x',b',y'\}$ and the triad $\{u,x,y\}$ of $M \ba a,b$.
  So $u=u'$.
  We assume that $s=x'$ in what follows, but the argument is the same when $s=y'$, only with the roles of $x'$ and $y'$ swapped.
  Since $u$ is $N$-flexible in $M \ba a,b,b'$, and $\{s,y'\}$ is a series pair in $M \ba a,b,b',u$, the matroid $M \ba a,b /y'$ has an $N$-minor.
  As $\{s,b'\}$ is a parallel pair in $M \ba a,b /y'$, it follows that $s$ is $N$-deletable in $M \ba a,b$.
  But since $\{s,x\}$ is a series pair in $M \ba a,b,b'$, the element $s$ is also $N$-contractible in $M \ba a,b$.
  So $s$ is an $N$-flexible element in $M \ba a,b$, but $s \notin \{u,x,y\}$, a contradiction.

  Now $u \notin \{u',x',y'\}$.
  If the gadget for $\{a,b'\}$ is Type~I, then, as $u$ is $N$-flexible in $M \ba a,b'$, we have $u \in \{u',x',y'\}$, a contradiction.
  So we may assume that the gadget for $\{a,b'\}$ is Type~II or Type~III.
  First assume that $u = z'$.
  Now $\{x',u',z'\}$ is a triangle, and recall that $\{b,s,x\}$ is a triad in $M \ba a,b'$.
  By orthogonality, if $s=x'$, then $u' \in \{b,x\}$, a contradiction.
  So $s=y'$.
  Since $u$ is $N$-flexible in $M \ba a,b,b'$,
  and $\{x',u'\}$ is a parallel pair in $M \ba a,b,b'/u$, the matroid $M \ba a,b,b',u'$ has an $N$-minor.
  In turn, $\{x',s\}$ is a series pair in $M \ba a,b,b',u'$, so $M \ba a,b/x'$ has an $N$-minor.  Finally, $\{b',s\}$ is a parallel pair in this matroid, so $s$ is $N$-deletable in $M \ba a,b$.  But then $s$ is $N$-flexible in $M \ba a,b$, a contradiction.

  Finally, assume that the gadget for $\{a,b'\}$ is Type~III and $u = w'$.
  As $M \ba a,b,u$ has an $N$-minor, and $\{b',u',x',y',z'\}$ is $2$-separating in $M \ba a,b,u$, with $b' \in \cocl_{M \ba a,b,u}(\{u',x',y'\})$, it follows from \cref{minor3conn} that $b'$ is $N$-contractible in $M \ba a,b$, a contradiction.
\end{subproof}

Recall that, when $M \ba b,b'$ is not 3-connected, it has either a triangle-triad, or one or two series pairs, by \cref{conn3}.
Now we address the subcase where $M \ba b,b'$ has two series pairs.

\begin{claim}
  \label{twoseriespairscase}
  Suppose $M \ba b,b'$ is $3$-connected up to series pairs.
  If $M \ba b,b'$ has two series pairs, then $u$ is not $N$-flexible in $M \ba a,b,b'$.
\end{claim}
\begin{subproof}
%
  Suppose $M \ba b,b'$ has distinct series pairs $\{x,s\}$ and $\{y,t\}$.
  Recall that $M \ba b'$ is $3$-connected, and $\{x,b,y\}$ is a triangle.
  Hence $\{b,x,y,s,t\}$ is $3$-separating in $M \ba b'$.
  Moreover, since $\{a,b,x,y,u\}$ is a cocircuit of $M$, we have $u \in \cocl_{M \ba a,b'}(\{b,x,y,s,t\})$.
  Let $F=E(M)-\{x,y,a,b,b',s,t,u\}$.
  Now $\lambda_{M \ba a,b',u}(F)=1$.
  We may assume that $M \ba a,b,b',u$ has an $N$-minor, for otherwise $u$ is not $N$-flexible in $M \ba a,b,b'$.
  Since $M \ba a,b,b'$ has an $N$-minor, and $\{x,s\}$ and $\{y,t\}$ are series pairs in this matroid, $s$ and $t$ are $N$-contractible in $M \ba a,b$.
  If $s = u$, say, then $\{x,u,b'\}$ is a triad of $M \ba b$, and hence also of $M \ba a,b$, contradicting that $\{u,x,y\}$ is the unique triad of $M \ba a,b$ containing $u$.
  So $s \neq u$ and, similarly, $t \neq u$.
  Recall that, as the gadget for $\{a,b\}$ is Type~I, the $(N,B)$-robust elements of $M \ba a,b$ are contained in $\{x,y,u\}$.
  Since neither $s$ nor $t$ is $(N,B)$-robust in $M \ba a,b$, we have $s,t \in B^*$.
  Now $(E(M)-F) \cap B = \{x,y,b'\}$, so $|F \cap B| = r(M)-3$.
  Since $E(M)-F$ contains the cocircuits $\{x,s,b,b'\}$, $\{y,t,b,b'\}$, and $\{x,y,u,a,b\}$, we have $r(F) \le r(M)-3$.
  So $r(F) = r(M)-3$ and $F \cap B$ spans $F$.

  We claim that $\{x,b,y\}$ is the unique triangle of $M \ba b'$ that meets $\{x,y\}$.
  Towards a contradiction, there is a triangle~$T$ of $M \ba b'$ distinct from $\{x,b,y\}$ that meets $\{x,y\}$.
  Suppose $x \in T$.
  Since $\{x,s,b\}$ is a triad of $M \ba b'$, the triangle~$T$ meets $\{s,b\}$ by orthogonality.
  But $\{x,b,y\}$ is closed, by \cref{easywin}, so $T=\{x,s,g\}$ for some $g \in E(M)-\{b,x,y,s,t,b'\}$.
  Moreover, as $M \ba b,b'$ is $3$-connected up to the two series pairs $\{x,s\}$ and $\{y,t\}$, the matroid $M \ba b,b'/x,y$ is $3$-connected.
  But $\{s,g\}$ is a parallel pair in this matroid, a contradiction.
  A similar argument applies when $y \in T$.

  Suppose there exist $p \in F \cap B$ and $q \in F \cap B^*$ such that the pivot $A^{pq}$ is allowable.
  Let $B' = B \triangle \{p,q\}$.
  By \cref{nostronginbasis}, $q$ is not $(N,B')$-strong.
  We claim that $p$ is also not $(N,B')$-strong.
  Towards a contradiction, suppose that $p$ is $(N,B')$-strong.
  Then, by \cref{strongprops}, $p$ is in a triad $T^*$ of $M \ba a,b$ that meets $\{x,y\}$, and $(T^*-p) \cup e$ is a triangle-triad in $M \ba f,p$ for some $\{e,f\} = \{a,b\}$.
  As $(T^*-p) \cup e$ is a triangle of $M$ that meets $\{x,y\}$, this triangle is $\{x,b,y\}$.
  So $(T^*-p) \cup e = \{x,b,y\}$ is a triangle-triad in $M \ba a,p$.
  But then $\{x,y,p\}$ is a triad in $M \ba a,b$.
  Since $p \neq u$, as $p \in F$ but $u \notin F$, the set $\{u,p,x,y\}$ is a cosegment in $M \ba a,b$, contradicting that the unique triad containing $u$ is $\{u,x,y\}$.
  So $p$ is not $(N,B')$-strong, as claimed.

  For any $q \in F \cap B^*$, as $F \cap B$ spans $q$, the entries in the column of $A$ labelled by $q$ are zero in any row labelled by an element of $B-F$, and, in particular, $A_{b'q} = A_{xq} = A_{yq} = 0$.
  Since each $q \in F \cap B^*$ is not a loop, there exists some $p \in F \cap B$ such that $A_{pq} \neq 0$.
  Similarly, for any $p \in F \cap B$ that is not a coloop of $M|F$, there exists some $q \in F \cap B^*$ such that $A_{pq} \neq 0$.
  By \cref{allowablenonxy}, the pivot $A^{pq}$ is allowable.
  If $q$ is $N$-contractible, or $p$ is $N$-deletable, then there are more $(N,B \triangle \{p,q\})$-robust elements than $(N,B)$-robust elements.
  As $p$ and $q$ are not $(N,B\triangle \{p,q\})$-strong, as shown in the previous paragraph, this contradicts the fact that $B$ is a bolstered basis.

  So each element of $F$ is either $N$-essential in $M \ba a,b$, or it is a coloop of $M|F$ that is in $B$, so it is $N$-deletable, and not $N$-contractible, in $M \ba a,b$.
%
  Let $Z$ be the set of coloops of $M|F = M \ba \{a,b,b',x,y,s,t,u\}$, and let $G=\{x,y,s,t\}$, so $(F,G)$ is a $2$-separation in $M \ba a,b,b',u$.
  Then 
  $Z \subseteq \cocl_{M\ba a,b,b',u}(G) \cap F$.
  So $r_{M \ba a,b,b',u}^*(Z) 
  \le r^*_{M \ba a,b,b',u}(G) + r^*_{M \ba a,b,b',u}(F) - r^*(M \ba a,b,b',u) = \lambda_{M \ba a,b,b',u}(G) \le 1$, by submodularity.
  Therefore $r_{M \ba a,b,b',u}^*(Z) \le 1$; that is,
  the coloops of $M|F$ are either coloops of $M \ba a,b,b',u$, or they are contained in a series class of $M \ba a,b,b',u$, where $M \ba a,b,b',u$ has an $N$-minor.
  In the former case the elements of $Z$ are $N$-flexible in $M \ba a,b,b',u$, so they are also $N$-flexible in $M \ba a,b$, a contradiction.
  So 
  $Z$ is contained in a series class $S$ of $M \ba a,b,b',u$.
  Say some $s \in Z$ is not $N$-essential in $M \ba a,b,b',u$.
  Recall that each element of $Z$ is $N$-deletable and not $N$-contractible in $M \ba a,b$, so $M \ba a,b,b',u,s$ has an $N$-minor.
  Then, in $M \ba a,b,b',u,s$, any element $s' \in S-s$ is a coloop, so it is $N$-flexible in $M \ba a,b$.
  Thus, if $|Z| \ge 2$, then $Z$ contains an $N$-flexible element, a contradiction. So
  $|Z| \le 1$.
  Now at most one element of $F$ is not $N$-essential in $M \ba a,b$, so
  $|F| \le |E(N)|+1$, and thus $|E(M)| = |F| +8 \le |E(N)| + 9$, a contradiction.
\end{subproof}

It remains only to address the subcase where $M \ba b,b'$ has a single series pair.  We break this into two parts, depending on whether or not $M \ba a,b'$ is $3$-connected.

\begin{claim}
  \label{singleseriespaircase1}
  Suppose $M \ba b,b'$ is $3$-connected up to series pairs.
  If $M \ba b,b'$ has a single series pair and $M \ba a,b'$ is not $3$-connected, then $u$ is not $N$-flexible in $M \ba a,b,b'$. 
\end{claim}
\begin{subproof}
  Suppose $M \ba b,b'$ has a single series pair $\{x,t\}$ and $M \ba a,b'$ is not $3$-connected.
  By \cref{conn2}, $M \ba a,b'$ is $3$-connected up to series classes, and every series pair of $M \ba a,b'$ avoids $\{x,y\}$.
  Let $\{j,k\}$ be such a series pair of $M \ba a,b'$.
  Since $M \ba a,b,b'$ has an $N$-minor, $j$ and $k$ are $N$-contractible in $M \ba a,b$.
  If $u = j$, say, then $\{b',u,k\}$ is a triad of $M \ba a$, and hence also of $M \ba a,b$, contradicting that $\{u,x,y\}$ is the unique triad of $M \ba a,b$ containing $u$.
  So $u \neq j$ and, similarly, $u \neq k$.
  By a similar argument, $u \neq t$.
  Thus, $j$, $k$ and $t$ are not $(N,B)$-robust in $M \ba a,b$, and hence $j,k,t \in B^*$.

  Let $F=E(M)-\{a,b,b',x,y,t,u,j,k\}$.
  Then $(E(M)-F) \cap B = \{b',x,y\}$, so $|B \cap F| = r(M)-3$.
  Note that $b' \notin \cocl_{M \ba a,b}(\{u,x,y\})$, since $\{u,x,y\}$ is the unique triad of $M \ba a,b$ containing $u$, so $\{u,x,y\}$ is a triad of $M \ba a,b,b'$.
  Since $M \ba a,b$ is $3$-connected, $\{j,k\}$ and $\{x,t\}$ are series pairs in $M \ba a,b,b'$.
  Now, due to the cocircuits $\{u,x,y\}$, $\{j,k\}$ and $\{x,t\}$ of $M \ba a,b,b'$, we have $r(F) \le r(M \ba a,b,b') - 3 = r(M)-3$.
  So $r(F) = r(M)-3$ and $B \cap F$ spans $F$.
  The argument that follows uses a similar approach to \cref{twoseriespairscase}, with some subtle, but important, differences.

  We claim that $\{x,b,y\}$ is the unique triangle of $M \ba b'$ that contains $x$.
  Towards a contradiction, suppose $x$ is in a triangle~$T$ of $M \ba b'$ distinct from $\{x,b,y\}$.
  Since $\{x,t,b\}$ is a triad of $M \ba b'$, the triangle~$T$ meets $\{t,b\}$ by orthogonality.
  But $\{x,b,y\}$ is closed, by \cref{easywin}, so $T=\{x,t,g\}$ for some $g \in E(M)-\{b,b',x,t\}$.
  Moreover, $M \ba b,b'$ is $3$-connected up to the single series pair $\{x,t\}$, so $M \ba b,b'/x$ is $3$-connected.
  But $\{t,g\}$ is a parallel pair in this matroid, a contradiction.

  Suppose there exist $p \in F \cap B$ and $q \in F \cap B^*$ such that the pivot $A^{pq}$ is allowable.
  Let $B' = B \triangle \{p,q\}$.
  By \cref{nostronginbasis}, $q$ is not $(N,B')$-strong.
  We claim that $p$ is also not $(N,B')$-strong.
  Towards a contradiction, suppose that $p$ is $(N,B')$-strong.
  Then, by \cref{strongprops}, $p$ is in a triad $T^*$ of $M \ba a,b$ such that $T^* \cap B'$ is a non-empty subset of $\{x,y\}$, and $(T^*-p) \cup e$ is a triangle-triad in $M \ba f,p$ for some $\{e,f\} = \{a,b\}$.
  If $x \in T^*$, then $(T^*-p) \cup e = \{x,b,y\}$, so $\{x,b,y\}$ is a triangle-triad in $M \ba a,p$.
  This implies that $\{x,y,p\}$ is a triad in $M \ba a,b$.
  But then, as $p \neq u$, since $p \in F$ but $u \notin F$, the set $\{u,p,x,y\}$ is a cosegment in $M \ba a,b$, contradicting that the unique triad containing $u$ is $\{u,x,y\}$.
  So $x \notin T^*$; therefore $T^* \cap B' = \{y\}$.
  Now $T^* = \{y,p,v\}$ is a triad of $M \ba a,b$, with $v \notin B'$. 
  Moreover, $v \neq u$, since the unique triad of $M \ba a,b$ that contains $u$ is $\{u,x,y\}$.
  Hence $\{u,x,y,p,v\}$ is a confining set relative to the basis $B'$, so $|E(M)| \le |E(N)| + 9$ by \cref{confiningset}, a contradiction.
  So $p$ is not $(N,B')$-strong, as claimed.

  For any $q \in F \cap B^*$, as $F \cap B$ spans $q$, we have $A_{b'q} = A_{xq} = A_{yq} = 0$; and, as $q$ is not a loop, there exists some $p \in F \cap B$ such that $A_{pq} \neq 0$.
  Moreover, for any $p \in F \cap B$ that is not a coloop of $M|F$, there exists some $q \in F \cap B^*$ such that $A_{pq} \neq 0$.
  By \cref{allowablenonxy}, the pivot $A^{pq}$ is allowable.
  If $q$ is $N$-contractible, or $p$ is $N$-deletable, then there are more $(N,B \triangle \{p,q\})$-robust elements than $(N,B)$-robust elements.
  As $p$ and $q$ are not $(N,B\triangle \{p,q\})$-strong, as shown in the previous paragraph, this contradicts the fact that $B$ is a bolstered basis.

  So each element of $F$ is either $N$-essential in $M \ba a,b$, or it is a coloop of $M|F$ that is in $B$, so it is $N$-deletable, and not $N$-contractible, in $M \ba a,b$.
  We may assume that $M \ba a,b,b',u$ has an $N$-minor, for otherwise $u$ is not $N$-flexible in $M \ba a,b,b'$.
  Now, in particular, each element of $F$ is either $N$-essential in $M \ba a,b,b',u$, or it is a coloop of $M|F$ that is $N$-deletable, and not $N$-contractible, in $M \ba a,b,b',u$.

  Let $Z$ be the set of coloops in $M|F = M \ba \{a,b,b',u,x,y,j,k,t\}$, and let $G=\{x,y,j,k,t\}$, so $(F,G)$ is a partition of $E(M \ba a,b,b',u)$.
  Then 
  $Z \subseteq \cocl_{M \ba a,b,b',u}(G) \cap F$.
  In $M \ba a,b,b',u$, the set $\{x,y,t\}$ is contained in a series class, and $\{j,k\}$ is a series pair, so $r_{M \ba a,b,b',u}^*(Z) \le r^*_{M \ba a,b,b',u}(G) + r^*_{M \ba a,b,b',u}(F) - r^*(M \ba a,b,b',u) = \lambda_{M \ba a,b,b',u}(G) \le 2$, by submodularity.
  If every element of $Z$ is $N$-essential in $M \ba a,b,b',u$, then every element of $F$ is $N$-essential in $M \ba a,b,b',u$, and so $|F| \le |E(N)|$, and thus $|E(M)| = |F| + 9 \le |E(N)| + 9$, a contradiction.
  So there exists some $f \in Z$ such that $M \ba \{a,b,b',u,f\}$ has an $N$-minor.
  Suppose there also exists some $f' \in Z-f$ such that $M \ba \{a,b,b',u,f,f'\}$ has an $N$-minor.
  If $\{f,f'\}$ is coindependent in $M \ba a,b,b',u$, then, as $\lambda_{M \ba \{a,b,b',u,f\}}(G) \le 1$ and $f'$ is not a coloop in $M \ba \{a,b,b',u,f\}$, we have $\lambda_{M \ba \{a,b,b',u,f,f'\}}(G)=0$, in which case $j$ and $k$ are $N$-flexible in $M \ba a,b$, a contradiction.
  So $r^*_{M \ba a,b,b',u}(\{f,f'\}) \le 1$.
  Then $f$ and $f'$ are $N$-contractible in $M \ba a,b,b',u$, 
  a contradiction.
  So every element of $Z-f$, and indeed $F-f$, is $N$-essential in $M \ba \{a,b,b',u,f\}$.
  Thus $|F-f| \le |E(N)|$, and $|E(M)| = |F| +9 \le |E(N)| + 10$,
  a contradiction.
\end{subproof}

\begin{claim}
  \label{singleseriespaircase2}
  Suppose $M \ba b,b'$ is $3$-connected up to series pairs.
  If $M \ba b,b'$ has a single series pair and $M \ba a,b'$ is $3$-connected, then 
  $u$ is not $N$-flexible in $M \ba a,b,b'$.
\end{claim}
\begin{subproof}
  Suppose $M \ba b,b'$ has a single series pair $\{x,t\}$ and $M \ba a,b'$ is $3$-connected.
  If $M \ba a,b'$ is $N$-fragile, then $u$ is not $N$-flexible in $M \ba a,b'$ as required. So we may assume that $M \ba a,b'$ is not $N$-fragile and thus there is a gadget for $\{a,b'\}$.
  Let this gadget be $(\{x',y'\},u')$ if Type~I, $(x',y',u',z')$ if Type~II, or $(x',y',u',z',w')$ if Type~III.
  There are two cases: 
  \begin{enumerate}
    \item[(i)] $a$ blocks in the gadget, and $b' \in \cl(\{x',y'\})$, or
    \item[(ii)] $b'$ blocks in the gadget, and $a \in \cl(\{x',y'\})$.
  \end{enumerate}

  Since $M \ba b$ and $M \ba b'$ are $3$-connected, $\{b,b',x,t\}$ is a cocircuit of $M$.
  Note that $\{x,t\}$ is a series pair in $M \ba a,b,b'$.
  Since this matroid has an $N$-minor, $t$ is $N$-contractible in $M \ba a,b$.
  Clearly $t \notin \{b,x\}$.
  Moreover, since $M \ba a,b$ is $3$-connected, and $\{u,x,y\}$ is the unique triad containing $u$, we have $t \notin \{a,u,y\}$.
  Thus, as $t$ is not $(N,B)$-robust, $t \in B^* - \{a,b\}$.

  Suppose we are in case (i), so $\{b',x',y'\}$ is a triangle of $M$.
  By orthogonality with the cocircuit $\{b,b',x,t\}$, we may assume $x' \in \{b,x,t\}$ (note that we lose no generality by making this assumption, since the argument that follows applies regardless of whether the gadget for $\{a,b'\}$ is Type~I, Type~II, or Type~III).
  Suppose $x' \in \{b,x\}$.
  In $M \ba b'$, the set $\{b,x,y\}$ is a triangle that meets the cocircuit $\{a,u',x',y'\}$.
  By orthogonality, these sets intersect in at least two elements.
  But if $y' \in \{b,x,y\}$, then $b' \in \cl(\{b,x,y\})$, contradicting \cref{easywin}.
  So $u' \in \{b,x,y\}$.
  Now $M \ba a,u'$ has the triangle-triad $\{b',x',y'\}$, so this matroid is not $3$-connected.  But this contradicts \cref{switchbxy} when $u' \in \{x,y\}$, or the fact that $M \ba a,b$ is $3$-connected when $u' = b$.

  Now we may assume that $x'=t$, so $\{b',t,y'\}$ is a triangle.
  Since $t$ is $N$-contractible, $y'$ is $N$-deletable in $M \ba a,b$.
  Suppose $y' \in \{a,b,x,y,u\}$.
  By orthogonality between the triangle $\{b',t,y'\}$, and the cocircuit $\{a,b,x,y,u\}$ of $M$, the set $\{b',t\}$ meets $\{a,b,x,y,u\}$, contradicting that $t \notin \{a,b,x,y,u\}$.
  We deduce that $y' \notin \{a,b,x,y,u\}$, so $y'$ is not $(N,B)$-robust.
  Thus $y' \in B - \{x,y\}$.

  Now $\{b',t,y'\}$ is a triangle with $b',y' \in B$.
  Then $A_{b't} \neq 0$ and $A_{xt}=A_{yt}=0$.
  By \cref{allowablenonxy}, the pivot~$A^{b't}$ is allowable.
  Let $B' = B \triangle \{b',t\}$.
  As $u$ is the only $(N,B)$-strong element in $M \ba a,b$, the $(N,B')$-strong elements of $M \ba a,b$ are contained in $\{u,b',t\}$.
  By \cref{nostronginbasis}, $t \in B'$ is not $(N,B')$-strong.
  Suppose that $b'$ is $(N,B')$-strong.
  Then, by \cref{strongprops}, $b'$ is in a triad $T^*$ of $M \ba a,b$ such that $T^* \cap B'$ is a non-empty subset of $\{x,y\}$, and $(T^*-b') \cup e$ is a triangle-triad in $M \ba f,b'$ for some $\{e,f\} = \{a,b\}$.
  By orthogonality between $T^*$ and the triangle $\{b',t,y'\}$, and since $y' \in B'-\{x,y\}$, we have $t \in T^*$.
  If $T^*=\{b',t,y\}$, then, as $\{b',t,x\}$ and $\{x,y,u\}$ are also triads of $M \ba a,b$, the set $\{b',t,y,x,u\}$ is a cosegment, contradicting that the unique triad containing $u$ is $\{u,x,y\}$.
  So $T^*=\{b',t,x\}$, in which case $\{t,x,e\}$ is a triangle-triad in $M \ba f,b'$ for some $\{e,f\} = \{a,b\}$.
  If $e =b$, then $\{x,b,y,t\}$ is a segment in $M$, contradicting \cref{easywin}.
  On the other hand, if $e=a$, then $\{t,x,a\}$ is a triangle-triad in $M \ba b,b'$, contradicting that $\{x,t\}$ is a series pair in $M \ba b,b'$.
  So $b'$ is not $(N,B')$-strong.
  Now $u$ is the only $(N,B')$-strong element of $M \ba a,b$, but $t$ is $(N,B')$-robust, so the number of $(N,B')$-robust elements is greater than the number of $(N,B)$-robust elements, contradicting that $B$ is a bolstered basis.

  Now suppose we are in case~(ii), so $\{a,x',y'\}$ is a triangle of $M$.
  Since $\{u,x,y,a,b\}$ is a cocircuit of $M$, by orthogonality $\{x',y'\}$ meets $ \{u,x,y,b\}$.

  Assume that $x' \in \{b,x,y\}$.
  In $M \ba b'$, the set $\{b,x,y\}$ is a triangle that meets the cocircuit $\{a,u',x',y'\}$.
  By orthogonality, these sets intersect in at least two elements.
  But if $y' \in \{b,x,y\}$, then $a \in \cl(\{b,x,y\})$, 
  contradicting \cref{easywin}.
  So $y' \notin \{b,x,y\}$ and hence $u' \in \{b,x,y\}$.
  Now $u' \neq b$ otherwise $M \ba b,b'$ has a triangle-triad $\{a,x',y'\}$, a contradiction.
  Therefore $u' \in \{x,y\}$.

  Suppose $x' = b$.
  By orthogonality between the triangle $\{x',y',a\}$ and the cocircuit $\{b,b',x,t\}$, we have $y' \in \{b',x,t\}$.
  But $y' \notin \{b,x,y\}$, so $y' = t$.
  Now, in $M \ba a,b'$, the unique triad containing $u'$ is $\{u',b,t\}$.
  But $\{x,b,t\}$ is a triad of $M \ba a,b'$, and $M \ba a,b'$ is $3$-connected, so $u'=x$.
  Now $\{a,b',x,b,t\}$ and $\{b,b',x,t\}$ are cocircuits of $M$, a contradiction.

  So $x' \in \{x,y\}$.
  Without loss of generality, let $u'=x$ and $x'=y$.
  In $M \ba a,b'$, the unique triad containing $u'=x$ is $\{u',x',y'\} = \{x,y,y'\}$.
  But $\{b,x,t\}$ is also a triad of $M \ba a,b'$, so $\{y',y\}=\{b,t\}$, a contradiction.

  Now we may assume that $x' \notin \{b,x,y\}$.
  Similarly, $y' \notin \{b,x,y\}$.
  So $x'$ or $y'$ is $u$.
  Without loss of generality let $x'=u$.  Then $y' \notin \{a,b,x,y,u\}$.
  Moreover, $u' \notin \{b,x,y\}$, for otherwise the triangle $\{b,x,y\}$ and the cocircuit $\{a,b',u',u,y'\}$ intersect in a single element, contradicting orthogonality.
  We may assume that $M \ba a,b,b',u$ has an $N$-minor, for otherwise $u$ is not $N$-flexible in $M \ba a,b,b'$.
  Since $\{x,t,b,b'\}$ is a cocircuit in $M$, we have $r^*_{M \ba a,b,b',u}(\{x,t\}) \le 1$.
  But if $t$ is a coloop in $M \ba a,b,b',u$, then it is $N$-flexible in this matroid and hence in $M \ba a,b$, a contradiction.
  So $\{x,t\}$ is a series pair in $M \ba a,b,b',u$.
  Moreover, $\{x,y\}$ is a series pair in $M \ba a,b,u$, and $b' \notin \cocl_{M \ba a,b,u}(\{x,y\})$, for otherwise $b'$ is $N$-contractible in $M \ba a,b$.
  So $\{x,y,t\}$ is contained in a series class of $M \ba a,b,b',u$.
  Thus $M \ba a,b,b',u / x,y$, and hence $M \ba a,b' / x,y$, has an $N$-minor.
  As $b$ is a loop in $M \ba a,b' / x,y$, it is $N$-flexible in $M \ba a,b' / x,y$ and hence in $M \ba a,b'$.
  Moreover, $M \ba a,b' / x,b$ has an $N$-minor, where $y$ is a loop in this matroid, so $y$ is also $N$-flexible in $M \ba a,b'$.
  Similarly, $M \ba a,b' / y,b$ has an $N$-minor, so $x$ is $N$-flexible in $M \ba a,b'$.
  Now each element in $\{x,y,b\}$ is $N$-flexible in $M \ba a,b'$, so these elements are in the gadget for $\{a,b'\}$.
  But $\{x,y,b\} \cap \{u',x',y'\} = \emptyset$, and there are at most two elements outside of $\{u',x',y'\}$ that are $N$-flexible in $M \ba a,b'$, a contradiction.
\end{subproof}

Recall that the theorem holds for any $b'\in B-\{x,y\}$ such that $M \ba b,b'$ is $3$-connected, so we may assume that $M \ba b,b'$ is not $3$-connected.
By \cref{conn3,triangletriadcase,twoseriespairscase,singleseriespaircase1,singleseriespaircase2}, $u$ is not $N$-flexible in $M \ba a,b,b'$.
By \cref{setup}, $b'$ is $N$-essential in one of $M \ba a,b,u$ and $M \ba a,b / u$. 
So the theorem indeed holds for any $b' \in B-\{x,y\}$.
\end{proof}

In order to improve the bounds obtained when applying this result, it would be desirable to relax the requirement in \cref{thegrandfantasy} that $|E(M)| \ge |E(N)|+11$, instead only requiring that $|E(M)| \ge |E(N)|+10$.
Although there are only two places in the proof where this bound is required, once each in the proofs of \cref{megagadgetcase,singleseriespaircase1}, 
whether such an improvement can indeed be made remains open.

\section{Final remarks}

When applying \cref{thegrandfantasy}, it may be useful to obtain a $3$-connected $N$-fragile minor of $M \ba a,b$; the next lemma enables one to do this.

\begin{lemma}
  \label{subfrag3conn}
  \hb\ Suppose $M$ has a pair of elements $\{a,b\}$ such that $M\ba a,b$ is $3$-connected with an $N$-minor, $|E(M)| \ge |E(N)| + 10$, and $M \ba a,b$ is not $N$-fragile, so \cref{bcosw-thm}(ii)(b) holds.
  Suppose the gadget for $\{a,b\}$ is Type~I, and $M \ba a,b,u/x$ is not $N$-fragile.
  Then either $M \ba a,b,u/x,y$ or $M \ba a,b,u/x \ba y$ is $3$-connected and $N$-fragile.
\end{lemma}
\begin{proof}
  Since $M \ba a,b,u/x$ is not $N$-fragile, and the gadget for $\{a,b\}$ is Type~I, the unique $N$-flexible element of $M \ba a,b,u/x$ is $y$.
  So $M \ba a,b,u/x,y$ and $M \ba a,b,u/x \ba y$ are $N$-fragile matroids.
  Assume that neither is $3$-connected.
  Then, by \cref{genfragileconn}, $y$ is in a triangle and a triad of $M \ba a,b,u/x$.
  By orthogonality, $y$ is an internal element of a $4$-element fan.  Let $s$ be the other internal element, so $\{s,y\}$ is contained in a triangle and a triad.
  Then $s$ is in a series pair in $M \ba a,b,u/x \ba y$, so $s$ is $N$-contractible in $M \ba a,b,u/x$.
  Also $s$ is in a parallel pair in $M \ba a,b,u/x,y$, so $s$ is $N$-deletable in $M \ba a,b,u/x$.
  So $s$ is $N$-flexible in $M \ba a,b,u/x$, a contradiction.
\end{proof}

We finish with a corollary that illustrates how we foresee \cref{thegrandfantasy} being used in arguments to bound the size of an excluded minor for $\mathbb{P}$-representability, for some partial field $\mathbb{P}$.
Clearly, an $N$-fragile matroid (or, indeed, any matroid having an $N$-minor) has at most $|E(N)|$ elements that are $N$-essential.
However, for certain classes we may hope for better than this: for example,
when $\mathbb{P} \in \{\mathbb{H}_5,\mathbb{U}_2\}$, it follows from the work of Clark, Mayhew, van Zwam, and Whittle~\cite{CMvZW16} that a $3$-connected $\mathbb{P}$-representable $\utfutf$-fragile matroid with at least nine elements has at most two $N$-essential elements, for $N \in \utfutf$ (we provide a proof of this in \cite{paper2}).

\begin{corollary}
  \label{gadgetsresult}
  \hb\ Suppose $M$ has a pair of elements $\{a,b\}$ such that $M\ba a,b$ is $3$-connected with an $N$-minor, $M$ has no triads, $|E(M)| \ge |E(N)| + 11$, and $M \ba a,b$ is not $N$-fragile, so \cref{bcosw-thm}(ii)(b) holds.
  Either $r(M) \le 6$, or there is an $N$-fragile minor~$M_1$ of $M$ with at least three $N$-essential elements and $|E(M_1)| \ge |E(M)|-5$, where $M_1$ is $3$-connected up to parallel classes.
\end{corollary}

\begin{proof}
  Suppose $r(M) \ge 7$.
  Up to swapping the labels on $a$ and $b$, we may assume that $a$ blocks in the gadget for $\{a,b\}$.
  By \cref{wmatype1}, either the gadget for $\{a,b\}$ is Type~I, or $M \ba a,x$ is $3$-connected with an $N$-minor, but not $N$-fragile, and the gadget for $\{a,x\}$ is Type~I.
  Then, up to swapping the labels on $b$ and $x$, we may assume that the gadget for $\{a,b\}$ is Type~I.
  Now, for each $b' \in B-\{x,y\}$, \cref{thegrandfantasy} implies that $b'$ is $N$-essential in either $M \ba a,b,u$ or $M \ba a,b / u$.
  Since $|B-\{x,y\}| \ge 5$, it follows that for some $M_0 \in \{M \ba a,b,u, M \ba a,b / u\}$, the matroid $M_0$ has at least three $N$-essential elements.

  First suppose $M_0 = M \ba a,b / u$.
  Since the $N$-flexible elements of $M \ba a,b$ are contained in $\{x,y,u\}$, the matroid $M_1$ is $N$-fragile for some $M_1 \in \{M_0, M_0/x, M_0/y, M_0/x,y\}$.
  By \cref{genfragileconn}, it follows that $M_1$ is $3$-connected up to parallel classes.
  Now suppose $M_0 = M \ba a,b,u$.
  Then, as $\{x,y\}$ is a series pair in this matroid, $M_0 / x$ has an $N$-minor.
  Moreover, since the only triad of $M \ba a,b$ containing $u$ is $\{u,x,y\}$, the matroid $M_0/x$ has no series pairs.
  If $M_0/x$ is $N$-fragile, then let $M_1=M_0/x$, so $M_1$ is $3$-connected up to parallel classes by \cref{genfragileconn}.
  Otherwise, we obtain a $3$-connected $N$-fragile minor $M_1$ of $M_0$ having $|E(M_1)| \ge |E(M)|-5$ by applying \cref{subfrag3conn}.

  Finally, observe that the $N$-essential elements of $M_0$ remain $N$-essential in the minor $M_1$, so $M_1$ has at least three $N$-essential elements.
\end{proof}

\bibliographystyle{acm}
\bibliography{lib}

\end{document}